\documentclass[10pt]{book}
\usepackage{amsmath}
\usepackage{amsthm}
\usepackage{amsfonts}
\usepackage{graphicx}
\usepackage{slashed}
\usepackage{mathrsfs}
\usepackage[all]{xy}
\usepackage{enumerate}
 \fontencoding{T1}
  \fontfamily{garamond}
  \fontseries{m}
  \fontshape{it}
  \fontsize{12}{15}
  \selectfont

\newcommand{\Nat}{\mathbb{N}} 
\newcommand{\Z}{\mathbb{Z}}

\newcommand{\R}{\mathbb{R}}
\newcommand{\C}{\mathbb{C}}

\newcommand{\M}{\mathscr{M}}
\newcommand{\A}{\mathscr{A}}
\newcommand{\B}{\mathscr{B}}
\newcommand{\G}{\mathscr{G}}

\newcommand{\F}{\mathscr{F}} 

\newcommand{\E}{\mathscr{E}}
\newcommand{\W}{\mathscr{W}}
\newcommand{\V}{\mathscr{V}}

\newcommand{\ima}{\mathrm{Im }}

\newcommand{\sD}{\slashed D}

\newcommand{\ti}{\times}

\newcommand{\ra}{\rightarrow}

\newtheorem{thh}{Theorem}[section]
\newtheorem{re}[thh]{Remark}
\newtheorem{de}[thh]{Definition}
\newtheorem{prop}[thh]{Proposition}

\newtheorem{lemma}[thh]{Lemma}
\newtheorem{coro}[thh]{Corollary}

\begin{document}
\titlepage
\selectfont{}
\vspace*{-2cm}
%\begin{center}
%	\begin{minipage}[c]{0.50\linewidth}
%		\raggedright \includegraphics[height=50px]{logo_amu.eps}
%	\end{minipage}\hfill
%\end{center}	

%	\begin{minipage}[c]{0.50\linewidth}
%		\raggedleft \includegraphics[height=50px]{int}
%	\end{minipage}\hfill 
%\end{center}
\begin{flushleft}
	\vspace{0.2cm}
	\LARGE AIX-MARSEILLE UNIVERSIT\'E\\
%	\begin{center}
%	\begin{minipage}[c]{0.30\linewidth}
%		\raggedright \includegraphics[height=50px]{logo_amu}
%	\end{minipage}\hfill
%\end{center}
%	
	\vspace{0.2cm}
	\Large ECOLE DOCTORALE EN MATH\'EMATIQUES ET INFORMATIQUE DE MARSEILLE - ED 184\\
	\vspace{0.2cm}
	\normalsize FACULT\'E DE SCIENCES ST. CHARLES \\
	\vspace{0.2cm}
	INSTITUT DE MATH\'EMATIQUES DE MARSEILLE, UMR 7373\\
    \begin{center}
		\vspace{1cm}
		\textbf{TH\`ESE DE DOCTORAT}\\
    \end{center}
	\vspace{0.5cm}
    Discipline: \textbf{Math\'ematiques}\\
    Sp\'ecialit\'e: \textbf{G\'eometrie}\\
    \begin{center}
        \vspace{0.5cm}
        \textbf{Diogo VELOSO}\\
        \vspace{0.8cm}
        \Large \bf{Seiberg-Witten theory on 4-manifolds with periodic ends}\\
    \end{center}
	\vspace{ 0.8cm}
    \normalsize Soutenue le 19 d\'ecembre 2014\\

\vspace{0.3cm}
  
    Rapporteurs:

\vspace{0.3cm}

\begin{tabular}{ ll  }
    	\textbf{ROLLIN, Yann} & Universit\'e de Nantes \\
    	\textbf{SAVELIEV, Nikolai} & University of Miami\\
\end{tabular}

\vspace{0.3cm}
Membres du jury :
\vspace{0.3cm}

\begin{tabular}{ ll  }
    	\textbf{ROLLIN, Yann} & Universit\'e de Nantes \\
       \textbf{DRUET, Olivier} & Universit\'e Lyon 1 \\
    \textbf{KOLEV, Boris} & Universit\'e d'Aix-Marseille \\      
       \textbf{TOMA, Matei} & Universit\'e de Lorraine  \\   
     \textbf{TELEMAN, Andrei} & Universit\'e d'Aix-Marseille  \\
\end{tabular}

\vspace{0.3cm}
Directeur de th\`ese
\vspace{0.3cm}

\begin{tabular}{ ll  } 
     \textbf{TELEMAN, Andrei} & Universit\'e d'Aix-Marseille  \\
\end{tabular}
\end{flushleft}

\usefont{T1}{bch}{m}{n}\selectfont{}

%\maketitle

\chapter*{Abstract}

In this thesis we prove fundamental analytic results  dedicated to  a new version of  Seiberg-Witten theory: the cohomotopical Seiberg-Witten theory  for a Riemannian, $\mathrm{Spin}^{c}(4)$ 4-manifold with periodic ends,  $(X,g,\tau)$  . Our results show that, under certain technical assumptions on $(X,g,\tau)$,  this new version is coherent and leads to Seiberg-Witten type invariants  for this new class of 4-manifolds.

\vspace{1 pc}
First,  using Taubes Fredholmness criteria for end-periodic operators on manifolds with periodic ends, we show that, for a Riemannian 4-manifold with periodic ends $(X,g)$, verifying certain topological  conditions, the Laplacian $\Delta_+:L^{2}_{2}(\Lambda^{2}_{+})\ra L^{2}(\Lambda^{2}_{+})$ is a Fredholm operator.  This  allows us to prove an important Hodge type decomposition for positively weighted Sobolev 1-forms on $X$. 

Next we prove, assuming non-negative scalar curvature on  each end and certain technical topological conditions,  that the associated Dirac operator associated with an end-periodic connection (which is ASD at infinity) is Fredholm.

In the second part of the thesis we establish an isomorphism between between the de Rham cohomology group, $H^{1}_{dR}(X,i\R)$ (which is a topological invariant of $X$) and the harmonic group intervening in the above Hodge type decomposition of the space of positively weighted Sobolev 1-forms on $X$.  We also prove two important short exact sequences relating the gauge group of our Seiberg-Witten moduli problem and the cohomology group $H^{1}(X,2\pi i\Z)$.

In the third part, following ideas of Kronheimer-Mrowka in the compact case, we prove our main results: the coercivity of the Seiberg-Witten map and compactness of the moduli space for a  4-manifold with periodic ends $(X,g,\tau)$ verifying  the above conditions.

Finally, using our coercitivity  property, we show (using the formalism developed by Okonek-Teleman) that a Seiberg-Witten type  cohomotopy invariant associated to $(X,g,\tau)$ can be defined.  Explicit computations and general properties of this new invariant  (which require radically different techniques)  will be considered in future articles.

\vspace{0.5cm}
Keywords : 4-manifold, Seiberg-Witten theory, Dirac operator, Fredholm operator, cohomotopic invariant.

\vspace{1 pc}
AMS math classification: 47A53, 53C07, 53C27, 57R57, 58D27

\chapter*{Resum\'e}

Dans cette th\`ese on prouve des r\'esultats analytiques fondamentaux sur une nouvelle version de la th\'eorie de Seiberg-Witten: la th\'eorie cohomotopique  de Seiberg-Witten pour des 4-vari\'etes Riemanniennes  $\mathrm{Spin}^{c}(4)$ a bouts p\'eriodiques, $(X,g,\tau)$. Nos r\'esultats montrent, que sur certaines conditions techniques en $(X,g,\tau)$, cette nouvelle version est coh\'erente et m\`ene a des invariants de Seiberg-Witten pour cette classe de vari\'etes.

\vspace{1 pc}
Premi\`erement, en utilisant le crit\`ere de Fredholmit\'e de Taubes pour des operateurs p\'eriodiques dans des vari\'etes a bouts p\'eriodiques, on montre que pour une 4-variet\'e Riemmanienne a bouts p\'eriodiques $(X,g)$ v\'erifiant certaines conditions topologiques, le Laplacian $\Delta_+:L^{2}_{2}(\Lambda^{2}_{+})\ra L^{2}(\Lambda^{2}_{+})$ est un op\'erateur de Fredholm.  Cela nous permet de prouver une importante d\'ecomposition de type Hodge pour des 1-formes de Sobolev de $X$, a poids positif. 

Ensuite on prouve, en assumant certaines conditions topologiques et courbure scalaire non-negative sur les bouts, que l'op\'erateur de Dirac associ\'e a une connection p\'eriodique  (ASD a l'infini) est Fredholm.

Dans la deuxi\`eme partie de la th\`ese on d\'emontre un isomorphisme entre le groupe de cohomologie de de Rham $H^{1}_{dR}(X,i\R)$,  et le groupe harmonique intervenant dans la decomposition de Hodge des 1-formes Sobolev de $X$ a poids positif.  On prouve aussi l'existence de deux s\'equences exactes courtes liant le groupe de jauge de l'espace de modules de Seiberg-Witten  et le groupe de cohomologie $H^{1}(X,2\pi i\Z)$.

Dans la troisi\`eme partie, en utilisant des id\'ees de Kronheimer-Mrowka pour le cas compact, on prouve les principaux r\'esultats: la coercitivit\'e de l'application de Seiberg-Witten et la compacit\'e de l'espace de moduli pour une 4-variet\'e a bouts p\'eriodiques $(X,g,\tau)$, v\'erifiant les conditions mentionn\'ees plus haut.

Finalment, utilisant la coercivit\'e, on montre employant le formalisme de Okonek-Teleman, l'existence d'un invariant cohomotopique de type Seiberg-Witten type associ\'e a $(X,g,\tau)$.  Calculs explicits et propriet\'es de ce nouveau invariant (qui exigent techniques radicalement diff\' erents)  s\'eront consider\'es dans futurs articles.

\vspace{0.5cm}
Mots cl\'es: 4-variet\'e, th\'eorie de Seiberg-Witten, op\'erateur de Dirac, op\'erateur de Fredholm, invariant cohomotopique.

\chapter*{Remerciements}

Je tiens a remercier d'abord a mon directeur de th\`ese, Andrei Teleman, pour son support, sa patience et sa pers\'ev\'erance, sans lui cette th\`ese n'aurait pas \'et\'e possible.

Ma gratitude a Yann Rollin et a Nikolai Saveliev, pour l'honneur qui me conc\`edent en acceptant d'\^etre rapporteurs de cette th\`ese.

Je remercie aussi a Olivier Druet, Boris Kolev et a Matei Toma qui ont accept\'e de faire parte du jury.

Je remercie a tout le personnel du CMI et de l'I2M (ancien LATP) pour m'avoir accueilli et permis un bon d\'eroulement de ma th\`ese.. En sp\'ecial  aux membres du groupe AGT (Analyse G\'eometrie et Topologie) et a tout le personnelle administratif qui permet que la machine marcher en particulier a Marie Christine Tort et a Nelly Samut.

Je remercie a mes coll\`egues doctorants, en particulier a Saurabh, Rima, Slah, Chady, Laurent, Kaidi, Nhan, Arash, Benjamin... et tous les autres!

Finalement je veux exprimer ma gratitude a tous ceux qui me sont proches, en particulier a ma famille, et qui m'ont soutenu dans tous les moments de cette th\`ese.

\tableofcontents

\chapter{Introduction}\label{intro}

 \section{The Seiberg-Witten equations}
 
 \def\Herm{\mathrm{Herm}}
 
 The Seiberg-Witten equations have been introduced in the seminal article \cite{W}, and found rapidly spectacular applications in  4-dimensional differential topology. Compared with  the ASD equations used in Donaldson theory \cite{DK}, the Seiberg Witten equations have an Abelian symmetry group, and yield compact moduli spaces. Therefore, one of the hardest difficulties in Donaldson theory (the construction of a natural compactification of the ASD moduli spaces) does not occur at all in Seiberg-Witten theory. We recall briefly the
Seiberg-Witten  equations and the  definition of the Seiberg-Witten moduli spaces. 
 
Let $(X,g)$ be an oriented Riemannian 4-manifold, and denote by $P_g$ the  $\mathrm{SO}(4)$-bundle of orthonormal frames of $X$ which are compatible with the orientation. Endow  $(X,g)$  with a  $Spin^{c}(4)$ structure $\tau:Q\to P_g$, and denote by $\det(Q)$, $\Sigma^\pm$ the determinant line bundle, respectively the spinor bundles of $\tau$. The Seiberg-Witten map associated with $(X,g,\tau)$ is the map 
$$SW:\A(\det Q)\ti \mathscr{C}^{\infty}(X,\Sigma^+)\ra \mathscr{C}^{\infty}(X,\Sigma^-)\ti  \mathscr{C}^{\infty}(\Herm_0(\Sigma^+))$$
given by
\begin{equation}\label{SWMAP}
  SW(A,\varphi)=(\sD_A\varphi, \Gamma(F_{A}^+)-(\varphi\otimes\varphi)_0),
  \end{equation}
  where $\Gamma:i\Lambda^2_+\to \Herm_0(\Sigma^+)$  is the bundle isomorphism induced by $\tau$. For a self-dual form $\eta\in   \mathscr{C}^{\infty}(i\Lambda^2_+)$ we define the $\eta$-perturbed  Seiberg-Witten map by
  
  \begin{equation}\label{SWETA}
  SW_\eta(A,\varphi)=(\sD_A\varphi, \Gamma(F_{A}^++\eta)-(\varphi\otimes\varphi)_0).
  \end{equation}
  
 The map $SW_\eta$ is equivariant with respect to the natural actions of the gauge group $\G:=\mathscr{C}^{\infty}(X,S^1)$   on the spaces  
 $$\A:=\A(\det Q)\ti \mathscr{C}^{\infty}(X,\Sigma^+),\ \mathscr{C}^{\infty}(X,\Sigma^-)\ti  \mathscr{C}^{\infty}(\Herm_0(\Sigma^+)).$$
Denote by $\A^*$ the open subspace of  $\A$ consisting of pairs with non-trivial spinor component, and by $\B$,  $\B^*$ the quotients of $\A$,   respectively $\A^*$  by  the gauge group $\G$. After suitable Sobolev completions, the projection $\A^*\to \B^*$ becomes a principal $\G$-bundle. The ($\eta$-perturbed) Seiberg-Witten moduli space associated with $(X,g,\tau)$ is the quotient
$$\M:=SW^{-1}(0)/\G\subset \B, \M_\eta:=SW^{-1}_\eta(0)/\G\subset \B.
$$
One of the fundamental results in classical Seiberg-Witten theory is the following compactness theorem:
\begin{thh} For any perturbation form $\eta\in\mathscr{C}^{\infty}(i\Lambda^+)$ the moduli space  $\M_\eta$ is compact.
\end{thh} 
The proof uses the Weitzenb\"ock formula and the maximum principle to get first an a priori $\mathscr{C}^0$-bound of the spinor component, and then standard elliptic bootstrapping techniques.\\
The  moduli space of  ($\eta$-perturbed)  {\it irreducible} monopoles is the open  subspace
 $$\M^*:=\M\cap \B^*, \ \M^*_\eta:=\M_\eta\cap \B^*
 $$
 of $\M$ (respectively $\M_\eta$).
Note that $\M^*$ ($\M^*_\eta$) can be identified with the zero locus of  the section  $sw$ ($sw_\eta$) induced by $SW$ ($SW_\eta$) in the bundle
 $$\A^*\times_\G (\mathscr{C}^\infty(\Sigma^-)\oplus \mathscr{C}^\infty(\Herm_0(\Sigma^+))
 $$
 over $\B^*$, which becomes Fredholm after suitable Sobolev completions.
 
Using the infinite dimensional version of Sard theorem, one can prove that for a "generic" perturbation form $\eta$, the   map $SW_\eta$ is submersive at any {\it irreducible} vanishing point, in particular the moduli space $\M^*_\eta$ is smooth and has ``the expected dimension"
$$w_\tau:=\frac{1}{4}\left[c_1(\det(Q))^2-2e(X)-
3\sigma(X)\right]=\mathrm{index}_\R(sw_\eta)
$$
 at any point.  Fixing an orientation of the  line  $\det(H^1(X,\R))\otimes \det(H^2_+(X,\R))^\vee$, one can define an orientation of the smooth manifold   $\M^*_\eta$ for any such   generic perturbation form.
 
 Moreover, one can prove that the space of perturbations $\eta$ for which  reducible solutions appear  in the moduli space ( i.e. for which $\M^*_\eta\ne \M_\eta$) has codimension $b_+(X)$ in the space of perturbations, hence its complement is connected when $b_+(X)\geq 2$. 
 
 The Seiberg-Witten invariants  are defined following the standard pattern  used in many gauge theories: evaluate a canonical cohomology class on the fundamental class of a moduli space. In our case, for  a 4-manifold with $b_+(X)\geq 2$ and a  cohomology class $c\in H^*( \B^*,\Z)$ we put
 $$SW(X,c):=\langle c,[\M_\eta]\rangle
 $$
 where $\eta$ has been chosen such that  $\M^*_\eta= \M_\eta$ and $SW_\eta$ is  submersive at any   vanishing point.  The definition can be adapted to the case $b_+(X)=1$, but in this case one obtains, for any class $c\in H^*( \B^*,\Z)$,  two invariants $SW^\pm(X,c)$ which are related by a wall crossing formula \cite{OT1}.

 \section{Generalizations of Seiberg-Witten equations}\label{intro2}

Soon after the birth of Seiberg-Witten theory, several authors have introduced interesting generalizations
of the theory on non-compact manifolds. The first contribution in this direction is due to Kronheimer-Mrowka \cite{KM2}, who studied the Seiberg-Witten equations on AFAK (asymptotically flat almost K\"ahler) 4-manifolds, proving an interesting result about the  finiteness of the set of homotopy classes
of semi-fillable contact structures on 3-manifolds. Similar versions of  the Seiberg-Witten equations on non-compact manifolds have been studied by Mrowka-Rollin \cite{MR} and  Biquard \cite{Bi}, who used these equations on manifolds with finite volume  conical ends  to prove an unicity theorem for a class of complete Einstein metrics  on a class of quotients of the complex hyperbolic space.

An important direction in the development of the Seiberg-Witten theory on non-compact manifolds concern the class of 4-manifolds with cylindric ends, and the relations between the 4-dimensional Seiberg-Witten invariants for such manifolds and the Seiberg-Witten  Floer invariants for 3-manifolds.  Fundamental contributions in these directions are the remarkable monographs of  Nicolaescu \cite{N}, Kronheimer-Mrowka \cite{KM} and Fr\o yshov \cite{F2}, \cite{F3}. These developments have been inspired by Yang-Mills Floer theory, whose main analytic tool is the theory of the ASD equation on 4-manifolds with cylindrical ends (see \cite{MMR}, \cite{T2}, \cite{D}).  \\

An important remark: in Donaldson theory we have not  only  an extension  of the theory of the ASD equation on manifolds with cylindrical ends  but also an interesting extension, due to Taubes, on manifolds with {\it periodic} ends,  \cite{Taubes}. Taubes introduced and studied the ASD equation on this class of non-compact manifolds, and used the resulting moduli spaces to prove a spectacular theorem concerning the cardinality of diffeomorphism classes of exotic $\R^4$'s. 

To our knowledge, up till now the Seiberg-Witten equations on 4-manifolds with {\it periodic} ends  have not been studied yet. This thesis can be viewed as an attempt to fill this gap. Note, however, that Ruberman-Saveliev \cite{RS1}, \cite{RS2}, and Mrowka-Ruberman-Saveliev \cite{MRS} studied and used end-periodic Dirac operators on 4-manifolds with periodic ends to define new Seiberg-Witten type invariants for closed 4-manifolds with $b_+=0$.\\

Recently, going in a different direction,    Furuta  and Bauer found an interesting and efficient  refinement of the classical Seiberg-Witten theory on {\it closed} 4-manifolds. This new theory originated in Furuta's idea to use ``finite dimensional" approximations of the Seiberg-Witten map (instead of the moduli spaces of its zeroes) to define invariants. The invariants obtained using the  ``stable homotopy class" defined by the system of finite dimensional approximations of the Seiberg-Witten map are called cohomotopy Seiberg-Witten invariants, or Bauer-Furuta invariants.  A new version of these invariants (which are better adapted for manifolds with $b_1>0$ and $b_+=1$) has been constructed later by Okonek-Teleman [OT2].
We recall briefly the formalism of  [OT2]:

\def\dR{\mathrm{dR}}

Fix a basis connection $A_0\in  \A(\det(Q))$  and denote by $\A_0(\det(Q))$ the affine subspace
$$\A_0(\det(Q)):=A_0+Z^1_{\dR}(X,i\R).
$$
of $\A(\det(Q))$. In other words, $\A_0(\det(Q))$ is the affine subspace of connections on $\det(Q)$ having the same curvature as $A_0$.  The space $\A(\det(Q))$ can be written as
$$\A(\det(Q))=\A_0(\det(Q))+d^*(\Omega^2_+(i\R)), 
$$
hence any connection $A\in \A(\det(Q))$ can be written in a unique way as $A'+ v$ with $A'\in \A_0(\det(Q))$ and $v\in  \V:=d^*(\Omega^2_+(i\R))$.

Denote by $\G_{x_0}\subset\G$ the kernel of the evaluation map $\mathrm{ev}_{x_0}: \G\to S^1$, and note that $\G_{x_0}$ acts freely on $\A(\det(Q))$ leaving  invariant its affine subspace 
 $\A_0(\det(Q))$. The quotient  $T:=\A_0(\det(Q))/\G_{x_0}$  can be identified with $H^1(X,i\R)/H^1(X,2\pi i\R)$, hence is a torus of dimension $b_1(X)$. The projection 
 $$\A_0(\det(Q))\to T
 $$
 can be regarded  as a principal $\G_{x_0}$-bundle. Letting $\G_{x_0}$ act on the spaces $\mathscr{C}^{\infty}(\Sigma^\pm)$ in the natural way, we obtain associated complex vector bundles
 $$\E:=\A_0(\det(Q))\times_{\G_{x_0}} \mathscr{C}^{\infty} (\Sigma^+),\ \F:=\A_0(\det(Q))\times_{\G_{x_0}} \mathscr{C}^{\infty}(\Sigma^-) $$
 on $T$. Put  now $\W:= \mathscr{C}^{\infty}(\Herm_0(\Sigma^+))\simeq \Omega^2_+(X,i\R)$. With these notations we see that the Seiberg-Witten map descends to an $S^1$-equivariant map over $T$ 
\begin{equation}\label{SWoverT}
\begin{array}{c}
\unitlength=1mm
\begin{picture}(20,12)(-5,-4)
\put(-6,4){$\V\times \E$}
\put(5,5){\vector(2,0){10}}
\put(16,4){$ \W\times \F$}
\put(7.5,6.5){$SW$}
%\put(2,2){$\simeq$}
\put(2,2){\vector(2, -3){5}}
\put(9,-8){$T$}
%\put(-4,-3){$\lambda$}
\put(18,2){\vector(-2, -3){5}}
\put(16,-8){.}
\end{picture} 
\end{array} 
\end{equation}
\vspace{2mm}

After suitable Sobolev completions (which make $\V$, $\W$ Hilbert spaces, and $\E$, $\F$ Hilbert bundles over $T$), the linearization of this map at the zero section can be written as $(d^+,\sD)$, where $d^+:\V\to \W$ is a linear embedding with cokernel $i\mathbb{H}^2_+$, and  $\sD$ is a family of complex Fredholm operators parameterized by $T$.  This $S^1$-equivariant map over $T$ is used in [OT2] to define the Seiberg-Witten cohomotopy invariants. The fundamental analytic property of $SW$ which allows the construction of these invariants is its  coercivity property:
\begin{thh}\label{CoercOT}  For every constant $c>0$  there exists $C_c\geq 0$ such that     the implication%
$$\|SW(v,e)\| \leq c\Rightarrow \|(v,e)\|\leq C_c 
$$
holds for pairs $(v,e)\in \V\times\E$. 
\end{thh}

In [OT2] the authors have also pointed out that the map (\ref{SWoverT}) over $T$ can be obtained in a simpler way: one can  replace $\A_0(\det(Q))$ by a finite dimensional affine  subspace of  $\A(\det(Q))$ and the gauge group $\G_{x_0}$ by a discrete group. Let $\mathbb{H}^1\subset Z^1_\dR(X,i\R)$  be the imaginary harmonic space of $X$ and
$$G:=\{\theta\in \G|\ \theta^{-1} d\theta\in \mathbb{H}^1\}=\{\theta\in \G|\  d^*(\theta^{-1} d\theta)=0\}.
$$
The map  $\theta\mapsto [\theta^{-1} d\theta]_{\dR}$ defines an epimorphism $p:G\to 2\pi i H^1(X,\Z)$, hence a short exact sequence
$$1\to S^1\to G\stackrel{p}{\longrightarrow} 2\pi i H^1(X,\Z)\to 0.
$$
It's easy to see that replacing $\A_0(\det(Q))$ by $A_0+ \mathbb{H}^1$ and $\G_{x_0}$ by the discrete group
$$G_{x_0}:=\ker[\mathrm{ev}_{x_0}: G\to S^1]\simeq  2\pi i H^1(X,\Z)$$
 in the construction $\E$, $\F$ and    (\ref{SWoverT}), one obtains  equivalent objects. This construction shows that the Hilbert bundles  $\E$, $\F$ come with a natural \textit{flat} connections.\\
 
 Our results will deal with the following natural
 \\
 \\
 {\bf Question:} {\it Can one generalize this construction to the framework of 4-manifolds with periodic ends?  Does the obtained bundle map $SW$ satisfy a similar coercivity condition?}\\
 
Note that, for a manifold  $X$ with periodic ends,   $H^1(X,\R)$, $H^1_{c}(X,\R)$ can be both infinite dimensional, so it is not clear at all what space will play the role of the harmonic space $\mathbb{H}^1$.

\section{The results}

\def\Eg{\mathfrak{E}}
\def\End{\mathrm{End}}

By definition (see section  \ref{periodicend} for details), a manifold with periodic ends has a finite set of ends which will be denoted by $\Eg$. For every $e\in \Eg$ we have an open submanifold $\End_e(X)\subset X$ representing the end $e$, which can be obtained as an infinite union $\cup_{i\in\Nat} W_{e,i}$, where $W_{e,i}$ are copies of the same connected, open manifold $W_e$ with two ends: a positive and a negative end. The union $\End_e(X)=\cup_{i\in\Nat} W_{e_i}$ is obtained by glueing the positive end of each $W_{e_i}$ to the negative end of $W_{e_{i+1}}$. Identifying the two ends of  $W_e$ one obtains a closed manifold  $Y_e$ with $b_1(Y_e)\geq 1$, which will be called the ``generalized torus" corresponding to the end $e$.  For every $e\in \Eg$ the open manifold $\End_e(X)$ comes with a natural map $p_e:\End_e(X)\to Y_e$  which identifies  $\End_e(X)$ with a ``half-cyclic cover" of $Y_e$ . 

Note that our conditions will use essentially the fixed periodic ends structure and are not intrinsically associated to the manifold $X$.

 A Riemannian manifold  with periodic ends is a Riemannian manifold $(X,g)$, where $X$ is  a manifold with periodic ends, and $g$ is end-periodic, i.e. for every $e\in \Eg$ the  restriction of $g|_{\End_e(X)}$ coincides with $(p_e)^*(g_e)$ for a metric $g_e$ on $Y_e$. A $Spin^c(4)$ Riemannian 4-manifold with periodic ends is a triple $(X,g,\tau)$, where $(X,g)$ is  a Riemannian 4-manifold  with periodic ends, and $\tau$ is an end periodic $Spin^c(4)$-structure on $(X,g)$, i.e., it  is endowed with    fixed isomorphisms $\tau|_{\End_e(X)}=(p_e)^*(\tau_e)$ for   $Spin^c(4)$-structures $\tau_e:Q_e\to P_{g_e}$ on the generalized tori  $Y_e$.

\begin{center}
\includegraphics[scale=0.5]{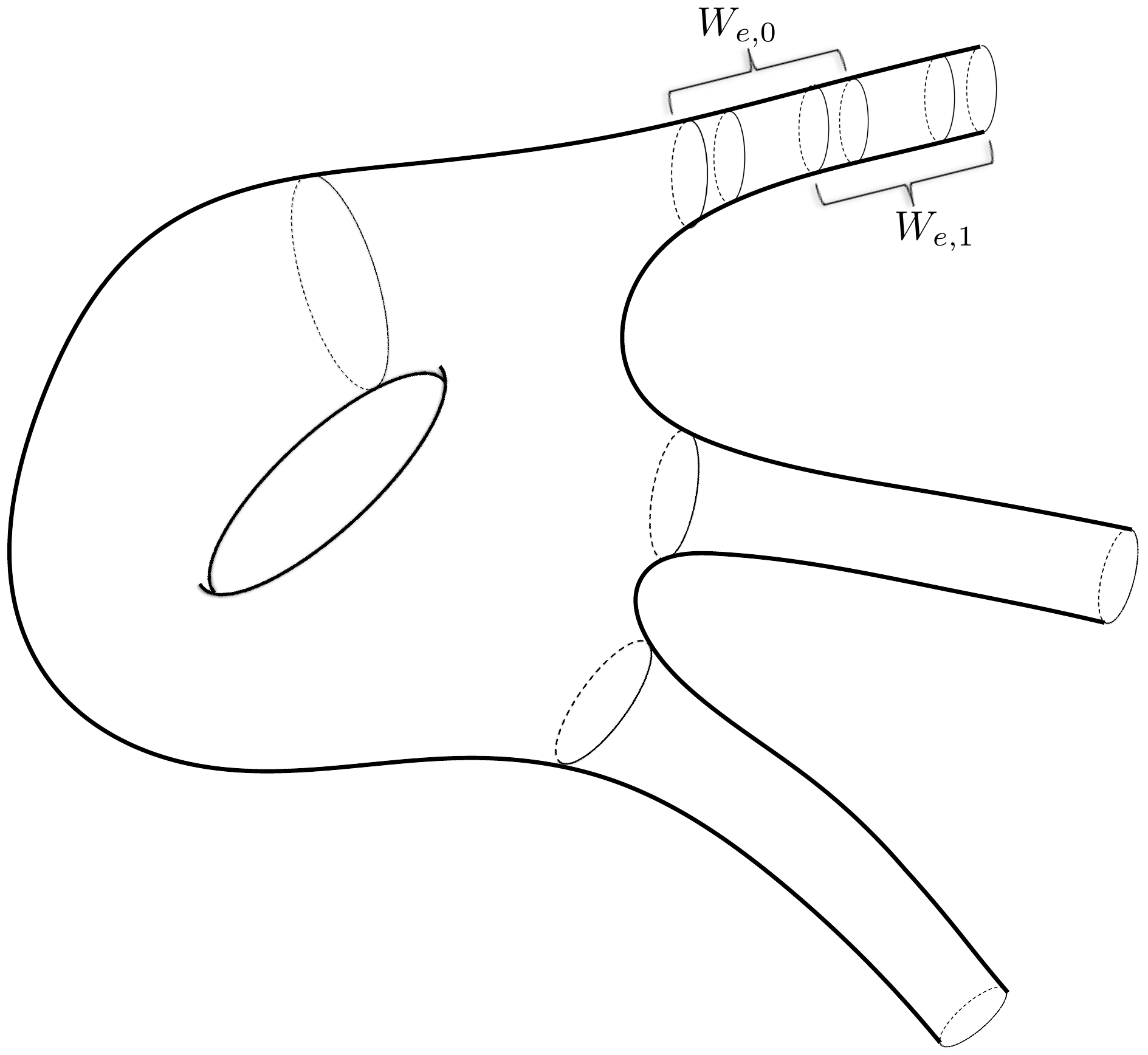}
\end{center}

Our main results concern $Spin^c(4)$ Riemannian 4-manifold $(X,g,\tau)$ with periodic ends satisfying the following conditions
\begin{enumerate}
\item  \label{TopoEnds} (Topological conditions on the periodic ends) For  every $e\in\Eg$ one has
\begin{enumerate}
\item  $b_+(Y_e)=0$,
\item $H_1(W_e,\Z)$ is torsion.
\end{enumerate}
\item  \label{TopoSpin} (Topological conditions on the $Spin^c(4)$-structure on the ends) For every $e\in \Eg$ one has
$$c_1(\det(Q_e))^2+b_2(Y_e)=0.
$$
\item  \label{sg} (Riemannian condition on the ends) For every $e\in \Eg$  the scalar curvature $s_{g_e}$ is non-negative on $Y_e$.
\end{enumerate}
If one assumes that $H_1(W_e,\Z)$ is finitely generated, then the condition $H_1(W_e,\Z)$ is torsion means simply $b_{1}(W_e)=0$. Note  that the condition $H_1(W_e,\Z)$ is torsion implies $b_1(Y_e)=1$. We will explain now our results, pointing out the role of each of these conditions  in our arguments. \\

We start by studying the ``Fredholmness"   of the relevant operators. The first result in this direction is
\begin{thh}  \label{FredDelta} Suppose that  $(X,g)$ is a Riemannian 4-manifold   with periodic ends such that $b_+(Y_e)=0$ and $H_1(W_e,\Z)$ is torsion for any $e\in\Eg$. Then  the operator
$$\Delta_+:L^{2}_{k+1}(i\Lambda^2_+)\to L^{2}_{k-1}(i\Lambda^2_+)
$$
is Fredholm. In particular, there exists $\epsilon >0$ such that for every $w\in (-\epsilon, \epsilon)$ the operator 
$$\Delta_+:L^{2,w}_{k+1}(i\Lambda^2_+)\to L^{2,w}_{k-1}(i\Lambda^2_+)
$$
is Fredholm and its kernel  $\mathbb{H}^2_+$  is independent of $w\in (-\epsilon, \epsilon)$.
\end{thh}

The result follows from Taubes's Fredholmness criterion (see \cite{Taubes} Lemma 4.3). In order to prove that the hypothesis of this criterion is satisfied we will  need   the following vanishing theorem: For any $\xi \in\C^*$ consider  the  locally constant sheaf (local coefficient system) $\C_\xi$ on $Y_e$  which is obtained using the \'etale cover $W_e\to Y_e$, the constant sheaf $\C$ on $W_e$  and the  automorphy factor $\xi$.  Applying a \v{C}ech type computation in the  \'etale topology of $Y_e$ we will show that $H^1(Y_e,\C_\xi)=0$ for any $\xi\in S^1\setminus\{1\}$.  \\

Using this result we will obtain a Hodge type decomposition theorem in dimension 1. For any $e\in\Eg$, let $\rho_e$ be a smooth real function on $X$ which vanishes on the complement of $X\setminus\End_e(E)$  and is constant 1  on $\End_e(E)\setminus W_{e,0}$. Put
$$F^\Eg:=\big\{\sum_{e\in\Eg} \rho_e \theta_e|\ \theta_e\in\R\big\}\simeq\R^\Eg.
$$
With this notation we will prove:
\begin{thh} \label{HodgeDec1}  In the conditions of Theorem \ref{FredDelta}, there exists $\epsilon>0$ such that for every $w\in(0,\epsilon)$ we have an $L^2$-orthogonal direct sum decomposition
$$L^{2,w}_k(i\Lambda^1)=d(L^{2,w}_k(i\Lambda^0)\oplus iF^{\Eg})\oplus \mathbb{H}^1_w\oplus d^*(L^2_{k+1,w}(i\Lambda^2_+)),
$$
where $\mathbb{H}^1_w$ is defined by
$$\mathbb{H}^1_w:=\ker [(d^*,d^+):L^{2,w}_k(i\Lambda^1)\to L^{2,w}_{k-1}(i\Lambda^0)\oplus L^{2,w}_{k-1}(i\Lambda^2_+)]=
$$
$$=\ker [(d^*,d):L^{2,w}_k(i\Lambda^1)\to L^{2,w}_{k-1}(i\Lambda^0)\oplus L^{2,w}_{k-1}(i\Lambda^2)]\ ,
$$
is finite dimensional, and can be identified  with $H^1(X,i\R)$.
\end{thh}

Note that the first order operator 
$$D:=(d^*,d^+):L^{2}_k(i\Lambda^1)\to L^{2}_{k-1}(i\Lambda^0)\oplus L^2_{k-1}(i\Lambda^2_+) $$
although elliptic, is certainly non-Fredholm when $\Eg\ne\emptyset$.  

This Hodge decomposition theorem plays a crucial role in our arguments.  
Our second Fredholmness result concerns the Dirac operator. 
\begin{thh}\label{FredDirac} With the notations and under the assumptions above suppose that $b_+(Y_e)=0$, $H_1(W_e,\Z)$ is torsion,  $c_1(\det(Q_e))^2+b_2(Y_e)=0$, and $s_{g_e}\geq 0$ for any $e\in\Eg$. Fix a connection $A_0\in \A(\det(Q))$ which is end-periodic and ASD on the complement of a compact set. Then the Dirac operator 
$$\sD_{A_0}: L^{2}_k(\Sigma^+)\to L^{2}_{k-1}(\Sigma^-)$$
 is Fredholm.
\end{thh}

The proof uses  Taubes's Fredholmness criterion again and the  Weitzenb\"ock formula on the generalized  tori $Y_e$. The condition $c_1(\det(Q_e))^2+b_2(Y_e)=0$ is needed to assure that the index of the induced Dirac operators on  the generalized tori $Y_e$ vanishes.\\

Using  these preparations we will prove the following fiberwise coercivity theorem:
\begin{thh} Suppose that the three conditions \ref{TopoEnds}, \ref{TopoSpin}, \ref{sg} above are verified and let $\epsilon>0$ satisfy the conditions in  Theorem \ref{FredDelta}, Theorem \ref{HodgeDec1}.  Fix $w\in (0,\epsilon)$, $k\geq 2$ and $A_0\in \A(\det(Q))$ which is end-periodic and ASD outside a compact set. Let 
 $(v_n, \phi_n)_n$ be a sequence in $d^*(L^2_{k+1,w}(i\Lambda^2_+))\times L^{2,w}_k(\Sigma^+)$ such that
$$\|SW(A_0+v_n, \phi_n)\|_{L^{2,w}_{k-1}}\leq C,
$$
for a positive constant $C>0$. Then
\begin{enumerate} 
\item   the sequence $(v_n,\phi_n)_n$ has a subsequence which is bounded  in $L^{2,w}_{k}$,
\item for every $\lambda\in (0,w)$ the sequence $(v_n,\phi_n)_n$  has a subsequence which is convergent in
$L^{2,\lambda}_{k-1}$.
\end{enumerate} 
\end{thh}

The result is obtained in 4 steps using ideas from (\cite{KM}, section II) : weak $L^{2}_1$ convergence,  strong $L^{2}_1$-convergence, strong $L^{3,\lambda}_1$ convergence and bootstrapping.   The   "input" of this sequence of arguments is an $L^{2}_1$-boundedness theorem which will be obtained using an adapted version of the "energy identity" of \cite{KM}.

Note that, without the condition $s_{g_e}\geq 0$ our method does not work.  Moreover, without this condition this coercivity statement does not hold even on   manifolds with cylindrical ends (see   \cite{KM}, \cite{N}, \cite{F1}).

\begin{coro} If $(\psi,\chi)\in L^{2,w+\varepsilon}_{k+1}(\Sigma^-)\times L^{2,w+\varepsilon}_{k+1}(\Herm_0(\Sigma^+))$, then  the fiber
$$\big\{(v,\phi)\in d^*(L^2_{k+1,w}(i\Lambda^2_+)\times  L^{2,w}_k(\Sigma^+)|\ SW(A_0+v, \phi)=(\psi,\chi)\big\}
$$
is compact in $L^{2,w}_k$.
\end{coro}

Fixing a compact set $\Pi\subset\mathbb{H}^1_w$, similar results can be obtained for sequences $(z_n,\phi_n)$ where $z_n=h_n+v_n$ with $v_n\in  d^*(L^2_{k+1,w}(i\Lambda^2_+))$ and $h_n\in\Pi$.  
In particular taking $(\psi,\chi)=(0,0)$ we obtain a compactness theorem for the $SW$ moduli space on $4$-manifolds with periodic ends satisfying our conditions.

We will end the thesis giving (for manifolds with periodic ends satisfying our conditions) the explicit construction  of  an  $S^1$-equivariant Seiberg-Witten map over a torus, which satisfies the  coercivity property needed in the construction of the cohomotopy invariants.

\chapter{4-manifolds with periodic ends}

In this chapter we are going to define $4$-manifolds with periodic ends and periodic structures on it, such as a metric or a $Spin^{c}(4)$-structure. We prove that, under certain conditions, the following two operators are Fredholm, the Laplacian $\Delta_{+}: L^{2}_{k+2}\ra L^{2}_{k}$ and the Dirac operator $\sD_{A}:L^{2}_{k+1}\ra L^{2}_{k}$. For the first the conditions are of a topological nature (see \ref{fredholmdelta}), for the second they concern both topological and geometric properties (see \ref{Dirac}). We prove also a Hodge type decomposition for the weighted Sobolev space $L^{2,w}_{k}(i\Lambda^1)$ (see \ref{Hodge}).

\section{Manifolds with periodic ends}\label{periodicend}

%
%
%\noindent
%Let $Y$ be a connected oriented differentiable 4-manifold with $b_{1}(Y)=1$ and $b^{+}_{2}(Y)=0$.
%Connectedness of $W$ and the universal coefficient theorem imply that the canonical map $h:H^{1}(Y,\Z)\ra \mathrm{Hom}(H_{1}(Y,\Z),\Z)$, that to an element $\gamma\in H^{1}(W,\Z)$ associates $h(\gamma)(u)=\langle\gamma,u\rangle$ %, where  $u\in H_{1}(W,\Z)$, %
%is an isomorphism.
%Using the canonical isomorphism $\iota:\mathrm{Hom}(H_{1}(W,\Z),\Z)\rightarrow \mathrm{Hom}(\pi_{1}(W),\Z)$, then $h(\gamma)$ extends naturally to a morphism on $\pi_{1}(W)$.
%mapping an element $u\in\pi_{1}(W)$ to $\gamma(u)=\langle\gamma,[u]\rangle$. Since $\Z\cong H^{1}(W,\Z)\cong\mathrm{Hom}(\pi_{1}(W),\Z)$ then, up to a sign, there is a unique  $\gamma\in H^{1}(W,\Z)$, such that $h(\gamma)\circ \iota$ is a generator of $\mathrm{Hom}(\pi_{1}(W),\Z)$.
%The kernel $H$ of the morphism $h(\gamma)\circ \iota$ is a normal subgroup of $\pi_{1}(W)$. The  normal covering of $W$ associated
%to $H$ is called the .... cover of $W$, and we shall denote it by $\widetilde{W}$.
%
%\vspace{0.5 pc}
%\noindent
%2)Suppose that the Poincar\'e dual $\mathscr{D}_{W} (\gamma)\in H_{3}(W,\Z)$ has a representative which is a topologically embedded hypersurface, $\Sigma$ with a neighborhood that 
%
%\vspace{0.5 pc}
%\noindent
%3)Then $\pi_{0}(W\setminus \Sigma)=\ast$ 
%
%\vspace{0.5 pc}
%\noindent
%4) If $b_{1}(\Sigma^+)=0$ then $H^{1}(\widetilde{W},\Z)=0$
%

\vspace{1 pc}
 
\begin{de}[Taubes]\label{periodic}
An  oriented differentiable  $n-$dimensional manifold  with one periodic end is a connected $n$-manifold $X$ endowed with the following structure:
\begin{enumerate}
\item a smooth connected oriented open $n-$manifold $W$ with a compact set $C$ such that $W\setminus C$ has two connected  components $N_{+}$ and $N_{-}$,
\item a compact set $C_+\subset N_+$ such that
\begin{enumerate}
\item  $N_+\setminus C_+$ has two connected components $N_{++}$ and $N_{+-}$,
\item  $W\setminus C_{+}$ is the disjoint union of $N_{-}\cup C\cup N_{+-}$ with $N_{++}$. 
\end{enumerate}
\item  a compact set $C_{-}\subset N_{-}$ such that 
\begin{enumerate}
\item $N_{-}\setminus C_{-}$ is the the disjoint union of $N_{--}$ and $N_{-+}$ ,
\item  $W\setminus C_{-}$ is the disjoint union of $N_{--}$ and $N_{-+}\cup C\cup N_{+}$. 
\end{enumerate}
\item a diffeomorphism $i:N_{+}\ra N_{-}$ which is orientation preserving and takes $N_{++}$ to $N_{-+}$ and $N_{+-}$ and to $N_{--}$.
\item an open subset $K\subset X$  with a connected open set $N\subset K$ such that $K\setminus N$ is compact, 
\item a compact set  $C_0\subset N$  such that $N\setminus C_0$ is the disjoint union of two open sets $N_{0-}$, $N_{0+}$ and such that $K\setminus C_0$ has two components $(K\setminus N)\cup N_{0-}$, and $N_{0+}$,
\item  a diffeomorphism $i_{-}:N\ra N_{-}$ with $i_{-}(N_{0-})=N_{--}$ and $i_{-}(N_{0-})=N_{-+}$.
\item an orientation preserving diffeomorphism 
$$\Phi:X\ra K\cup_{\{i_{-}:N\ra N_{-}\}}W\cup_{\{i:N_{+}\ra N_{-}\}} W\cup_{\{i:N_{+}\ra N_{-}\}} W\cup\cdots$$

\end{enumerate}

\end{de}

\def\End{\textrm{End}}
\def\Aut{\textrm{Aut}}

We will denote by  $W_i\subset X$ the inverse image (via $\Phi$) of the $i$-th copy of $W$  in the union above (for $i\geq 0$), and by $\textrm{End}(X)$ the union $\cup_{i\geq 0} W_i$. \\

Given an end periodic manifold $X$ we construct the associated "generalized torus" $Y$ to be the {\it closed} -$n$-manifold   $Y:=W\big /\sim_i$, where $\sim_i$ is the equivalence relation induced by the diffeomorphism $i:N_{+}\ra N_{-}$.

Note that  the closure $\overline\End(X)$  has an open neighborhood that can be naturally embedded in a cyclic cover of $\widetilde{Y}\to Y$, where 
$$\widetilde{Y}:=\cdots\cup_{\{i:N_{+}\ra N_{-}\}}W\cup_{\{i:N_{+}\ra N_{-}\}} W\cup_{\{i:N_{+}\ra N_{-}\}} W_{\{i:N_{+}\ra N_{-}\}}\cup\cdots
$$
is obtained by gluing in the obvious way a family of copies of $W$ parametrized by $\Z$.

We point out that Taubes's definition of a manifold with periodic ends is very general. In particular it is {\it not} required that $N$ is a tubular neighborhood of a smoothly embedded hypersurface. Adopting this general framework is essential for the topological results obtained in \cite{Taubes} and  which concern the exotic $\mathbb{R}^4$'s.

\vspace{1 pc}

We will see in section \ref{TwistedCo} that the 1-cohomology of $W$ and $Y$ can be related using a Mayer-Vietoris section (in the \'etale topology) which reads:

$$0\to H^0(Y,\Z)\to H^0(W,\Z)\stackrel{0}{\longrightarrow} H^0(N_+,\Z)\to H^1(Y,\Z)\to H^1(W,\Z)\to H^1(N,\Z)\to\dots
$$
Let $\gamma\in H^{1}(Y,\Z)$ be the image of $1\in H^0(N_+,\Z)$ in $H^1(Y,\Z)$. It is easy to see that $\gamma$, regarded as morphism $\pi_1(Y,p)\to \Z$ (for a point $p\in Y$) corresponds precisely  to the monodromy representation of the cover   $\widetilde{Y}\to Y$. In other words the morphism $\pi_1(Y,p)\to \Z$ induced by $\gamma$ is surjective and the normal subgroup $H\subset \pi_1(Y,p)$ associated with this cover coincides with $\ker(\gamma)$.

 Let $\tau_{Y}:Y\ra S^1$  be a smooth map representing the homotopy class corresponding to $\gamma$ under the standard identification  $H^{1}(Y,\Z)\ra[(Y,p),(S^1,1)]$, (see \cite{spanier}, chapter 8, Lemma 10)). Since the pull-back of $\gamma$ to $W$ and $\widetilde{Y}$ vanishes, it follows that  $\tau_{Y}$ has a lift $\tau_{\widetilde{Y}}:\widetilde{Y}\to \R$ satisfying the identity
 $$\tau_{\widetilde{Y}} (T(y))=\tau_{\widetilde{Y}}(y)+1\ ,
 $$
 where $T$ stands for the positive generator of $\Aut_Y(\widetilde{Y})\simeq\Z$.  Fixing a point $q\in \tilde Y$ and imposing $\tau(q)=0$, this lift becomes uniquely determined. Choosing a closed form $\theta$ on $Y$ representing the class $\gamma$ in de Rham cohomology, one can obtain   $\tau_{\widetilde{Y}}$ directly as a primitive of the pull-back of $\theta$ to  $\widetilde{Y}$.

We have identified an  open neighborhood of $\overline{ \End}(X)\subset X$ with its image in $\widetilde{Y}$ via the natural map. Therefore $\tau_{\widetilde{Y}}$ induces a smooth function  on an open neighborhood of $\overline{\End}(X)$. Using a cut-off function which is identically 1 on $\overline{\End}(X)$ we obtain an $\R$-valued function $\tau$ on $X$ which agrees with $\tau_{\widetilde{Y}}$ on $\overline \End(X)$.

\begin{de} Let $X$ be a manifold with one periodic end. A Riemannian metric $g$ on $X$ will be called end-periodic if its restriction to $\End(X)$ coincides with the pull-back of a metric on $Y$. A Riemannian manifold with one periodic end is a manifold with one periodic end endowed with an end-periodic metric.
\end{de}

A slightly more general notion is obtained requiring only  that $g$ coincides with the pull-back of a metric on $Y$ on the complement of a compact set. This generalization is not important because, changing $K$ accordingly, we come to an end periodic metric in the sense of our definition.\\

\def\Eg{\mathfrak{E}}

In a similar way one can define a (Riemannian) manifold with $k$ periodic ends. Denoting by $\Eg$ the set of ends of such a manifold we obtain, for every $e\in \Eg$ associated object
$$\End_e(X),\ W_e,\ W_{e,i},\ Y_e, \widetilde{Y}_e,\ \tau_e: X\to\R
$$
as in the case of manifolds with one periodic end.  We also put 
$$\End(X):=\cup_{e\in\Eg}\End_e(X), \  \End_{e,k}(X):=\cup_{i\geq k} W_{e,i}\  \End_k(X):=\cup_{e\in\Eg} \End_{e,k}(X).$$

For the definition of the weighted Sobolev spaces in the next section we will need the function
\begin{equation}\label{tau}
\tau:=\sum_{e\in \Eg} \tau_e
\end{equation}
on $X$. This function is proper and the sub level sets $X^{\leq a}:=\tau^{-1}((-\infty,a])$ are compact for every $a\in\R$. 

We remark that our construction gives, for every end $e\in \Eg$, two fundamental systems of neighborhoods of $e$, namely:
$$\End_{e,k}(X),\ X^{>a}_e:=\{\tau_e\}^{-1}(a,\infty) \hbox{ for } a>0.
$$

Note that  $\End_{e_k}(X)$ can be easily described at a union of copies of $W_e$, but its closure  can be a very complicated subset of $X$. On the other hand, if $a$ is a regular value of $\tau_e$, $X^{>a}_e$ will be the interior of the manifold with boundary $X^{\geq a}_e$.

 \pagebreak 
\section{End-periodic differential operators}

This  section is based on the article of Taubes (\cite{Taubes} section 3).
Let $X$ be a manifold periodic ends, we say that a vector bundle $\pi:E\ra X$ is  end-periodic, if for each end $e\in\Eg$ the restriction of $\pi$ to $\End(X)_{e}$, has been identified with the pullback of a vector bundle $\pi_e:E_e{}\ra Y_{e}$. In a similar way, we say that a connection $\nabla: \mathscr{C}_{c}^{\infty}(X,E)\ra \mathscr{C}^{\infty}_{c}(X,\Lambda^{1}\otimes E)$ is end-periodic if on each end $e\in\Eg$ it coincides with the pullback of a connection $\nabla_e:\mathscr{C}^{\infty}_{c}(Y_e,E_e)\ra \mathscr{C}^{\infty}_{c}(X,\Lambda^{1}\otimes E_e)$ under the identification.

\begin{de}
Let $(X,g)$ be an oriented Riemannian manifold with periodic ends, and $p:E\ra X$ be an end-periodic Euclidean(Hermitian) vector bundle endowed with an end-periodic Euclidean(Hermitian) connection $\nabla$. For $w\in \R$, $k\in\Nat$ and $1<p<\infty$ we define the norm
\begin{equation}\label{Sobolev}
\|\varphi\|_{L^{p,w}_{k}(E)}=\sum_{l=0}^{k}\int_{X} |e^{w\tau}\nabla^{l} \varphi|^{p} \cdot d\mathrm{vol}_g,    \hbox{    \textrm{        $\forall$}$\varphi\in\mathscr{C}^{\infty}_{c}(E)$}
\end{equation}
where $\tau$ is as in (\ref{tau}).\ We will call the completion of $\mathscr{C}^{\infty}_{c}(E)$ with respect to $\|\cdot\|_{L^{p,w}_{k}(E)}$, the weighted Sobolev space of weight \textrm{w}, $L^{p,w}_{k}(X,g)(E,\nabla)$. 
\end{de}

Note that $\nabla^l$ is obtained by tensoring $\nabla$ and the underlying Levi-Civita connection on tensor powers of $\Lambda^1$,  in the appropriate manner.
We also remark that the above weighted Sobolev space are independent of the choices made in the definition of $\tau$.

Let $(X,g)$ be a Riemannian manifold with periodic ends and $\pi_E:E\ra X$ and $\pi_F:F\ra X$ be periodic vector bundles. A differential operator $\partial:\mathscr{C}^{\infty}_{c}(E)\ra \mathscr{C}^{\infty}_{c}(F)$ is end periodic, if for each end $e\in\Eg$ the restriction of $\partial$ to $\End(X)_e$ coincides with the pullback of a differential operator $\partial_{e}:\mathscr{C}^{\infty}_{c}(Y_e,E_e)\ra \mathscr{C}^{\infty}_{c}(Y_e,F_e)$.\\

From now on we suppose our end-periodic vector bundles to be Euclidean (resp. Hermitian) and endowed with a periodic Euclidean (resp. Hermitian) connection.\\
We say that the extension of an end-periodic elliptic partial differential operator of order $r$, $\partial:\mathscr{C}^{\infty}_{c}(X,E)\ra \mathscr{C}^{\infty}_{c}(X,F)$, to $L^{2,w}_{\bullet}$ is Fredholm, if the natural extensions $\partial: L^{2,w}_{k+r}(E)\ra L^{2,w}_{k}(F)$  are Fredholm for all $k\in\Nat$.

The following theorem of Taubes (\cite{Taubes} Lemma 4.3) gives a necessary and sufficient condition for the extension of an end-periodic elliptic partial differential operator to weighted Sobolev spaces to be Fredholm.

\begin{thh}\label{Taubesexact}
Let $(X,g)$ be an end-periodic Riemannian manifold. Suppose that $\partial:\mathscr{C}^{\infty}_{c}(X,E)\ra \mathscr{C}^{\infty}_{c}(X,F) $ is an end-periodic elliptic differential operator of order $r$, over $X$.  Then the natural extensions of $\partial$ to $L^{2,w}_{\bullet}$ are  Fredholm if and only if for all $e\in\Eg$, and all $\xi\in\C^\times$ with $|\xi|=e^{w/2}$, the associated (elliptic) complex
\begin{align}\label{cohogroups}
0\ra\mathscr{C}^{\infty}_{c}(Y_e,E_e)\xrightarrow{\partial_e(\xi)} \mathscr{C}^{\infty}_{c}(Y_e,F_e)\ra 0
\end{align}
has vanishing cohomology groups.
\end{thh}

Here $\partial_e(\xi)$ denotes the differential operator $e^{-\lambda\tau_e}\partial (e^{\lambda \tau_e}\cdot)$, where $e^\lambda=\xi$. It does not depends on $\tau_e$, but only on $d\tau_e$, which on $\End_e(X)$ coincides with the pullback of a closed 1-form on $Y_e$, $\gamma_e$.
We remark that the vanishing of the cohomology of the complex (\ref{cohogroups}), is equivalent to the condition that the extension of (\ref{cohogroups}) to suitable Sobolev spaces is an isomorphism.

Another important result of Taubes (\cite{Taubes} Theorem 3.1), gives conditions which assure that the extension of an end-periodic elliptic partial differential operator is Fredholm for all $w\in\R$, but a discrete set.

\begin{thh}\label{Taubesgeneric}
Let $(X,g)$ be an end-periodic Riemannian manifold. Suppose that $\partial:\mathscr{C}^{\infty}_{c}(X,E)\ra \mathscr{C}^{\infty}_{c}(X,F)$  is an end-periodic elliptic differential operator and that for each $e\in\Eg$,  the index of $\partial_e$ is trivial, and the map
\begin{align*}
\sigma_{\partial_e}([\gamma_e]):H^{0}(\partial_e)\ra H^{1}(\partial_e)
\end{align*}
is injective. Then for all $w\in\R$ but a discrete set, the extension of the operator $\partial$ to $L^{2,w}_{\bullet}$ is Fredholm.
\end{thh}

For the definition of the map $\sigma_{\partial_e}$ see (\cite{Taubes}, p. 373) . In what follows, it will be enough to know that in the case of a first order differential operator ${\partial_e}$, $\sigma_{\partial_e}(\gamma)$ is the  evaluation of the symbol of $\partial_e$ applied to the 1-form $\gamma_e$.

\section{The twisted cohomology of $Y$}\label{TwistedCo}

Let $(X,g)$ be a Riemannian manifold with periodic ends. The objective of this section is to prove, that under certain topological assumptions, the operator $\Delta_+:L^{2}_{k+2}(\Lambda^2_+ X)\ra L^{2}_{k}(\Lambda^2_+ X)$ is Fredholm.
%To do this we will indirectly prove the vanishing of cohomology appearing in Theorem \ref{Taubesexact}.
\vspace{1 pc}

Suppose for the moment that $X$ has only one end. 
As explained in section \ref{periodicend}, the generalized torus $Y$ is identified with the quotient $W/\sim_i$, where $\sim_i$ is the equivalence relation induced by the diffeomorphism $i:N_+\to N_-$.  The  canonical projection  $\pi:W\to Y$ is an \'etale cover of $Y$, and it can be used to compute the cohomology of a sheaf of Abelian groups $\mathcal{F}$ on $Y$.  The fiber product $W\times_YW$ decomposes as

$$W\times_Y W =\Delta\coprod N_+\coprod N_-,
$$
 where $\Delta \simeq W$ is identified with the diagonal of the fiber product,
$N_+$ is identified with the component $\{(w_+,i(w_+))|\
w\in N_+\}$ of the fiber product, and  $N_-$  is identified with the component $\{(w_-,i^{-1}(w_-))|\ w\in N_-\}$ of the fiber product. A 1-cocyle for the \v{C}ech cohomology  of the \'etale cover $\pi:W\to
Y$ is determined by its restriction to $N_+$.

The spectral sequence associated with this cover (see \cite{AM}, Corollary 8.15 p.107) has
$$E_\infty^{p,q}=E_2^{p,q}=\left\{\begin{array}{ccc}
\ker (u_q:H^q(W,\pi^*(\mathcal{F}))\to H^q(N_+,\pi^*(\mathcal{F}))) & \rm when & p=0\\ \\
H^q(N_+,\pi^*(\mathcal{F}))/u_q (H^q(W,\pi^*(\mathcal{F})))  & \rm when & p=1\\ \\
  0 & \rm when & p>1
\end{array}\right. ,
$$
where, for $k\in\mathbb{N}$ and $h\in H^k(W,\pi^*(\mathcal{F}))$ we put
$$u_k(h):={h}|_{N_+}- i^*({h}|_{N_-} ).
$$
In other words, $E_2^{p,q}$ is the p-th \v{C}ech cohomology group associated with the \'etale cover $\pi$ and the functor $(s:U\to Y)\mapsto H^q(U,s^*(\mathcal{F}))$. 
This yields a Mayer-Vietoris type long exact sequence
$$0\to H^0(Y,\mathcal{F})\to H^0(W,\pi^*(\mathcal{F}))\stackrel{u_0}{\longrightarrow} H^0(N_+,\pi^*(\mathcal{F}))\to H^1(Y,\mathcal{F})\to   $$
$$\to H^1(W,\pi^*(\mathcal{F}))\stackrel{u_1}{\longrightarrow} H^1(N_+,\pi^*(\mathcal{F}))\to H^2(Y, \mathcal{F})\to\dots\ .
$$
For a fixed twisting element $\xi\in\mathbb{C}^*$ we are interested  in the 1-cohomology of the locally constant sheaf $\mathbb{C}_\xi$ on $Y$ obtained by factorizing the constant sheaf $\mathbb{C}$ on $W$ by the $i$-covering isomorphism $(i,\xi^{-1}\mathrm{Id}_{\mathbb{C}})$. In this case we have obvious identifications
$$H^k(W,\pi^*(\mathbb{C}_\xi))=H^k(W, \mathbb{C}),$$
and, via these identifications, we have
$$ u_k(h)=h|_{N_+}-\xi i^*(h|_{N_-}).
$$
In particular, since $N_+$ and $W$ are connected, $u_0=(1-\xi)\mathrm{id}_\mathbb{C}:\mathbb{C}\to\mathbb{C}$, which shows that 
\begin{equation}\label{H0}
H^0(Y,\mathbb{C}_\xi)=\left\{\begin{array}{ccc}
\mathbb{C} &\rm when &\xi=1\\
0 &\rm when &\xi\ne 1
\end{array}\right. .
\end{equation}
Define now the morphisms $v_\pm:H^1(W,\mathbb{Z})\to H^1(N_+,\mathbb{Z})$ by
$$v_+(h):= h|_{N_+}\ ,\ v_-(h):= i^*(h|_{N_-}),
$$
and put 
$$K_\xi:=\ker\big((v_+\otimes\mathrm{id}_\mathbb{C})-\xi (v_-\otimes\mathrm{id}_\mathbb{C})\big)\subset  H^1(W,\mathbb{C}).
$$
For $\xi=1$ we obtain the exact sequence
\begin{equation}\label{H11}
0\to\mathbb{C}\to H^1(Y,\mathbb{C})\to K_1\to 0, 
\end{equation}
and for $\xi\ne 1$ we obtain the isomorphism
\begin{equation}\label{H12}
H^1(Y,\mathbb{C}_\xi)\stackrel{\simeq}{\longrightarrow} K_\xi. 
\end{equation}

Note now that $\mathbb{C}_\xi$ is the sheaf of parallel sections of a flat complex line bundle $(L^\xi, \nabla_\xi)$ on $Y$.  Therefore $H^*(Y,\mathbb{C}_\xi)$ can also be computed using the de Rham theorem for the flat linear connection $\nabla_\xi$. In other words, we can compute   $H^*(Y,\mathbb{C}_\xi)$ using the resolution 
$$0\longrightarrow \mathbb{C}_\xi\hookrightarrow \mathscr{E}^0_Y\stackrel{d_\xi}{\longrightarrow} \mathscr{E}^1_Y\stackrel{d_\xi}{\longrightarrow} \mathscr{E}^2_Y\stackrel{d_\xi}{\longrightarrow}\dots 
$$
of $\mathbb{C}_\xi$ by fine sheaves. Here $d_\xi$ stands for the de Rham operator associated with
$\nabla_\xi$, regarded as a sheaf morphism. We obtain canonical isomorphisms
$$H^k(Y,\mathbb{C}_\xi)=H^k(\mathscr{C}_\xi),
$$
where $\mathscr{C}_\xi$ is the Rham complex
$$0\longrightarrow A^0(Y,L^\xi)\stackrel{d_\xi}{\longrightarrow}A^1(Y,L^\xi)\stackrel{d_\xi}{\longrightarrow}A^2(Y,L^\xi)\stackrel{d_\xi}{\longrightarrow}\dots . \eqno{(\mathscr{C}_\xi^{\rm })}
$$
In particular we obtain an isomorphism
\begin{equation}\label{Rham} H^1(Y,\mathbb{C}_\xi)=\frac{\ker[d_\xi: A^1(Y,L^\xi)\rightarrow   A^2(Y,L^\xi) ]}{\textrm{Im }[d_\xi: A^0(Y,L^\xi)\ra A^1(Y,L^\xi)]}\ .
\end{equation}
Suppose now that $(X,g)$ is a Riemannian 4-manifold with periodic ends. For $\xi\in S^1$ consider the SD-elliptic complex %
$$0\longrightarrow  A^0(Y,L^\xi)\stackrel{d_\xi}{\longrightarrow }A^1(Y,L^\xi)\stackrel{d_\xi^+}{\longrightarrow }A^2_+(Y,L^\xi)\longrightarrow  0 \eqno{(\mathscr{C}_{\xi}^{+})}
$$
associated with the flat unitary connection  $\nabla_\xi$.
\begin{re} \label{d+} Suppose $\xi\in S^1$. Then 
$$\ker[d_\xi: A^1(Y,L^\xi)\rightarrow    A^2(Y,L^\xi) ]=\ker[d_\xi^+: A^1(Y,L^\xi)\rightarrow   A^2_+(Y,L^\xi)] \ ,
$$
in particular one has an obvious identification
$$H^1(\mathscr{C}_\xi)=H^1(\mathscr{C}_{\xi}^+).$$
\end{re}
\begin{proof}

The cohomology of the elliptic complex $(\mathscr{C}_\xi)$ can be computed using Hodge's theorem, i.e., identifying $H^k(\mathscr{C}_\xi)$  with the $k$-th harmonic space of  $(\mathscr{C}_\xi)$, $\mathbb{H}^k(Y,\mathbb{C}_\xi)$. 

 Suppose now that $\alpha\in \ker d_\xi^+$.  Then $d_\xi\alpha$ will be a $d_\xi$-exact anti-selfdual $L^\xi$-valued 2-form.   For $\xi\in S^1$, the flat connection $\nabla_\xi$ becomes a unitary connection, hence the adjoint $d_\xi^*$ can be computed using the usual formula involving $d_\xi$ and the Hodge star operator. Therefore, as in the non-twisted case, we conclude that  $d_\xi\alpha$ belongs to the second harmonic space of $(\mathscr{C}_\xi)$  and is  $d_\xi$-exact. By Hodge's theorem, this implies $d_\xi\alpha=0$.
\end{proof}

For $\xi\in S^1$ put now
\begin{align*}
\Delta_{+,\xi}&:= d_\xi^+ d_\xi^*: A^2_+(Y,L^\xi)\longrightarrow A^2_+(Y,L^\xi)\\
\mathbb{H}^2_{+,\xi}&:=\ker \Delta_{+,\xi}=\ker [d^*_\xi: A^2_+(Y,L^\xi)\to  A^1(Y,L^\xi)]
\end{align*}

\begin{prop}\label{dimb2} Let $(X,g)$ be a Riemannian 4-manifold with one periodic end. Assume that 
\begin{equation*}
K_\xi=\{0\} \ \hbox{        \textrm{                              }  $\forall$ $\xi\in S^1$}. \eqno{(C)}
\end{equation*}
Then $\dim \mathbb{H}^2_{+,\xi} $ is independent of $\xi\in S^1$.
\end{prop}

\begin{proof}
 Suppose that $(C)$ holds, then using (\ref{H0}), we obtain
$$H^0(\mathscr{C}_{\xi}^{+})=H^0(\mathscr{C}_{\xi})=H^0(Y,\mathbb{C}_\xi)\simeq\left\{\begin{array}{ccc}
\mathbb{C}&\rm when & \xi=1\ \ \\
0 &\rm when & \xi\ne 1\ .
\end{array}\right. .
$$  
On the other hand, it follows from Remark \ref{d+}, and the formulae (\ref{H11}), (\ref{H12}), that
$$H^1(\mathscr{C}_{\xi}^{+})=H^1(\mathscr{C}_{\xi})=H^1(Y,\mathbb{C}_\xi)\simeq\left\{\begin{array}{ccc}
\mathbb{C}&\rm when & \xi=1\ \ \\
0 &\rm when & \xi\ne 1\ .
\end{array}\right. .
$$
Since the index of the complex $(\mathscr{C}_{\xi}^+)$ is independent of $\xi$, the result follows. 
\end{proof}
Note that assuming condition $(C)$ holds for every end $e\in\Eg$, the above Proposition extends to case of Riemannian 4-manifolds with multiple periodic ends.

\begin{re} Condition $(C)$ holds automatically when $H^1(W,\R)=0$. Note that the weaker condition $b_1(Y)=1$ does not imply $(C)$. Indeed, when $b_1(Y)=1$ it follows $K_1=\{0\}$, but we have no control on $K_\xi$  for $\xi\in S^1\setminus\{1\}$. 
\end{re}

\begin{re} Suppose $H^1(N_+,\R)=0$ (which holds for instance when $N_+$ is a bicolar around an embedded locally flat rational 3-sphere). Then the following conditions are equivalent:
\begin{enumerate}
\item  $(C)$  holds,
\item $H_1(W,\R)=0$,
\item $b_1(Y)=1$
\end{enumerate}
\end{re}

%
%\begin{de}\label{admissible}
%We say that $(X,g)$,  a Riemannian 4-manifold with periodic ends is \textit{admissible} if each end $e\in\Eg$ verifies the following properties
%$H_{1}(W_e,\Z)$ is torsion and $b_{2}^{+}(Y_e)=0$.
%\end{de}

Notice that the condition $H_{1}(W_e,\Z)$ is torsion implies that $b_{1}(W_e)=0$, \textit{if} $H_{1}(W_e,\Z)$ is finitely generated. 

\begin{thh}\label{fredholmdelta}
Let $(X,g)$ be a Riemannian 4-manifold with periodic ends such that for each end $e\in \Eg$ the condition $(C)$ holds and $b^{+}_{2}(Y_e)=0$. Then the operator $\Delta_{+}:L^{2}_{k+2}(i\Lambda^{2}_{+}X)\ra L^{2}_{k}(i\Lambda^{2}_{+}X)$ is Fredholm, for any $k\in\Nat$.
\end{thh}

\begin{proof}
This follows theorem \ref{Taubesexact} applied to operator $\Delta_+$, together with Proposition \ref{dimb2} and the fact that the hypothesis $b_{2}^{+}(Y_e)=0$, is satisfied for every $e\in\Eg$.
\end{proof}

\section{Hodge decomposition}\label{hodgedeco}

In this section we will prove a Hodge type decomposition for the space of $1$-forms in admissible Riemannian 4-manifolds with periodic ends.
%
%\begin{lemma}
%Let $D:E\ti[0,\epsilon[\ra F$ a continuous family of Fredholm operators with $\dim \ker D_{t}=k$ for $t>0$. Then $\dim \ker D_{0}\geq k$.
%\end{lemma}
%
%\begin{proof}
%Let $E=\ker D_{0}\oplus E_{0}$  and $F=\textrm{Im }D_{0}\oplus F_{0}$ be topological decompositions.
%Consider the inclusion $\iota_{0}:E_0\hookrightarrow E$, the projection $\pi_{0}:F\ra \textrm{Im } D_{0}$ and the family of linear continuous maps $\pi_{0}\circ D_{t}\circ \iota_{0}:E_{0}\ra \textrm{Im } D_{0}$. Since for $t=0$  $\pi_{0}\circ D\circ \iota_{0}$ is an isomorphism, then for all $t$ sufficiently close to $0$ the operator $\pi_{0}\circ D_{t}\circ \iota_{0}$ is an isomorphism. Hence $D_{t}\circ\iota_{0}$ is injective. Therefore $\ker D_t\cap E_0=\{0\}$ so that $\dim \ker D_t\leq \dim\ker D_0$ for all $t$ sufficiently close to $0$.
%
%
%\end{proof}

For a subspace $A\subseteq L^{2,w}$ denote by $A^{\perp_{L^{2,-w}_k}}$ the annihilator (orthogonal complement) of $A$ in $L^{2,-w}_k$ with respect to the $L^2$  inner product.

\begin{lemma}\label{duality}
Let $(X,g)$ be a Riemannian manifold with periodic ends and $p:E\ra X$ and $q:F\ra X$ be periodic Euclidean(Hermitian) vector bundles over $X$. Suppose that
$P$ is a periodic differential operator of order $r$ that extends to a Fredholm map $P:L^{2,w}_{k+r}(E)\ra L^{2,w}_{k}(F)$. Then
\begin{align*}
\textrm{Im }[P:L^{2,w}_{k+2r}(E)\ra L^{2,w}_{k+r}(F)]= \ker[P^*:L^{2,-w}_{k+r}(E)\ra L^{2,-w}_{k}(F)]^{\perp_{L^{2,w}_{k+r}}}
\end{align*}
and thus
\begin{align*}
\mathrm{codim}[P:L^{2,w}_{k+2r}(E)\ra L^{2,w}_{k+r}(F)=\dim \ker[P^*:L^{2,-w}_{k+r}(E)\ra L^{2,-w}_{k}(F)]
\end{align*}
\end{lemma}

\begin{proof}
Let $I$ be the image of $P:L^{2,w}_{r}\ra L^{2,w}$. Since $P$ is Fredholm, $I$ is a closed subspace of $L^{2,w}$ of finite codimension. 
The dual of $L^{2,w}$ can be identified with $L^{2,-w}$, the duality pairing corresponding to the $L^2$ inner product.

The $L^2$ orthogonal complement of $I$ in $L^{2,-w}$, $I^{\perp_{L^{2,-w}}}$ coincides with its annihilator in the dual of $L^{2,w}$, via the above identification.
Let now $v\in L^{2,-w}$ be an element of $I^{\perp_{L^{2,-w}}}$, then 
\begin{equation}\label{supp}
\langle P(x),v\rangle_{L^2}=0 \hbox{ for all  } x\in L^{2,w}_{r}
\end{equation}
This implies that $P^{*}v=0$ as a distribution;, and thus $v\in \mathscr{C}^{\infty}$. Applying standard regularity theorems (see \cite{LM} formula (2.4) p. 420,  and \cite{E} Lemma 1.30) we know that $\ker P^*\cap L^{2,-w}=\ker P^*\cap L^{2,-w}_{k}$ for every $k\in\Nat$. Let $K$ be this kernel and let $I_k\subset L^{2,w}_k$ be the image of $P:L^{2,w}_{k+d}\ra L^{2,w}_k$. We show that $I_k$ coincides with the $L^2$-orthogonal complement of $K$ in $L^{2,w}_{k}$, $K^{\perp_{L^{2,w}_{k}}}$.  If $u\in K^{\perp_{L^{2,w}_{k}}}$ then, $u\in K^{\perp_{L^{2,w}}}$ which, we know, is $I$, and also  $u\in L^{2,w}_{k}$. Thus $u=P(x)$ where $x\in L^{2,w}$. Using regularity we get, from $u=P(x)\in L^{2,w}_k$, we get $u\in L^{2,w}_{k+d}$, and so $u\in I_k$.

\end{proof}

\begin{coro}\label{h+=h-}
Suppose that $(X,g)$ is an Riemannian 4-manifold with periodic ends such that for every end $e\in\Eg$ condition $(C)$ holds and $b^{+}_{2}(Y_e)=0$ . Then there exists $\varepsilon>0$  such that for all $w\in(0,\varepsilon)$
\begin{align*}
\mathbb{H}_{-w}^{+}=\mathbb{H}^{+}_{w}=\mathbb{H}^{+}_{0}.
\end{align*}
\end{coro}

\begin{proof}
From theorem \ref{fredholmdelta}, $\Delta_{+}:L^{2}_{k+2}(\Lambda^2_+)\ra L^{2}_{k}(\Lambda^2_+)$ is Fredholm, and thus it is also self-adjoint. In particular for all $w$ sufficiently close to zero the operator $P_w=e^{w\tau}\Delta_{+} (e^{-w\tau}\cdot)$ will be in a neighborhood of  $\Delta_{+}$ in $\mathscr{L}(L^{2}_{k+2}(\Lambda^2_+), L^{2}_{k}(\Lambda^2_+))$. It follows $P_w$ is also Fredholm and has the same index, zero. Now conjugating back via $e^{w\tau}$ we get that $\Delta_{+,w}=\Delta_{+}:L^{2,w}_{k+2}(\Lambda^2_+)\ra L^{2,w}_{k}(\Lambda^2_+)$ is Fredholm of index zero. 
 Thus for all $w$ in a certain neighborhood of the origin
\begin{align*}
0&=\dim\ker \Delta_{+, w}-\textrm{codim } \textrm{Im } \Delta_{+,w}\\
&=\dim\ker \Delta_{+,w}-\dim\ker \Delta_{+_{-w}}
\end{align*}
Since  for $w>0$ we have $\ker \Delta_{+,w}\subseteq \ker\Delta_{+_{0}}\subseteq \ker \Delta_{+_{-w}}$ we get  the equality.
\end{proof}

We shall denote by $\mathbb{H}^{2}_{+}$ the harmonic space $\mathbb{H}_{-w}^{+}=\mathbb{H}^{+}_{w}=\mathbb{H}^{+}_{0}$ for $w\in (-\varepsilon,\varepsilon)$.

\vspace{1 pc}

Let $(X,g)$ be a Riemannian 4-manifold with periodic ends. For each $e\in\Eg$, let $\rho_e$ be a (fixed) smooth function which equals zero outside $X_{e}^{\geq 0}$ and $1$ in $X_{e}^{\geq 1}$.  We shall denote by $F^{\Eg}$ the real vector space spanned by the set$\{\rho_e\}_{e\in\Eg}$. Equivalently one can use the alternative space $F^\Eg$ as in the introduction. Consider the following complexes
\begin{align*}
&L^{2,w}_{k}(i\Lambda^{1})\xrightarrow{d^*\oplus d^+}L^{2,w}_{k}(i\Lambda^{0}\oplus i\Lambda^{2}_{+})\xrightarrow{d+d^+}L^{2,w}_{k}(i\Lambda^{0}\oplus i\Lambda^{2}_{+})\\
&L^{2,w}_{k}(i\Lambda^{0}\oplus i\Lambda^{2}_{+})\oplus iF^{\Eg} \xrightarrow{d+d^*} L^{2,w}_{k}(i\Lambda^{1})\xrightarrow{d^*\oplus d^+}L^{2,w}_{k}(i\Lambda^{0}\oplus\Lambda^{2}_{+}).
\end{align*}

We define $D$ to be the operator $d^*\oplus d^+$. Its formal adjoint $d+d^{*}$ shall be denoted by $D^*$. When acting on the weighted Sobolev spaces $L^{2,w}_{k}$, we shall insert a lower $w$ on these operators, becoming, $D_{w}$ and $D^{*}_w$ respectively.

\begin{lemma}\label{laplaiso}
Let $(X,g)$ be a Riemannian 4-manifold with periodic ends such that for every $e\in\Eg$ condition $(C)$ holds and $b^{+}_{2}(Y_e)=0$. Then there exists $\varepsilon>0$ such that for any $w\in (-\varepsilon,\varepsilon)$ the map
\begin{align}\label{lapla2}
\Delta_{+}{}_{\big |\{\mathbb{H}^{2}_{+}\}^{\perp_{L^{2, w}_{k+1}}}}:\{\mathbb{H}^{2}_{+}\}{}^{\perp_{L^{2, w}_{k+1}}}\ra \{\mathbb{H}^{2}_{+}\}^{\perp_{L^{2, w}_{k-1}}}
\end{align}
is an isomorphism.
\end{lemma}

\begin{proof}
Recall that by the duality Lemma \ref{duality}
\begin{equation}
\begin{aligned}
\ima [\Delta_{+}:L^{2, w}_{k+1}(i\Lambda^{2}_{+})\ra L^{2, w}_{k-1}(i\Lambda^{2}_{+})]&=\ker [\Delta_{+}:L^{2,- w}_{k-1}(i\Lambda^{2}_{+})\ra L^{2,- w}_{k-3}(i\Lambda^{2}_{+})]^{\perp L^{2,w}_{k-1}}\\
\label{lapla3}&=\{\mathbb{H}^{+}_{-w}\}^{\perp L^{2,w}_{k-1}}=\{\mathbb{H}^{2}_{+}\}^{\perp L^{2,w}_{k-1}}.
\end{aligned}
\end{equation}
Therefore, the above map, (\ref{lapla2}), which is by construction is injective, it is also surjective.
\end{proof}

\begin{lemma}\label{lemma}
Let $(X,g)$ be a Riemannian 4-manifold with periodic ends such that for every $e\in\Eg$ condition $(C)$ holds and $b^{+}_{2}(Y_e)=0$. Suppose that $(\varphi,\alpha) \in L^{2,-w}_{k+1}(i\Lambda^0\oplus i\Lambda^{2}_{+})$ is such that $d\varphi+d^*\alpha\in L^{2,w}_{k}(i\Lambda^{1})$.
Then $\alpha\in L^{2,w}_{k+1}(i\Lambda^{2}_{+})$ and $\varphi\in L^{2,w}_{k+1}(i\Lambda^{0})\oplus iF^\Eg$.
\end{lemma}

\vspace{1 pc}

Since the opposite inclusion is obvious,  the lemma states that 
\begin{equation}\label{avanthodge}
\begin{aligned}
&\big\{D^*_{-w}\big( L^{2,-w}_{k+1}(i\Lambda^0\oplus i\Lambda^{2}_{+})\big)\}\cap \big\{ L^{2,w}_{k}(i\Lambda^{1})\}\\
&=D^*_w\big\{\big(L^{2,w}_{k+1}(i\Lambda^{0})\oplus iF^\Eg\big)\oplus L^{2,w}_{k+1}(i\Lambda^{2}_{+})\big\}
\end{aligned}
\end{equation} 
%\begin{equation}\label{avanthodge}
%\begin{aligned}
%&(L^{2,w}_{k+1}(i\Lambda^{0})\oplus iF^\Eg)\oplus L^{2,w}_{k+1}(i\Lambda^{2}_{+})\\
%&=\{(\varphi,\alpha) \in L^{2,-w}_{k+1}(i\Lambda^0\oplus i\Lambda^{2}_{+})|\ d\varphi+d^*\alpha\in L^{2,w}_{k}(i\Lambda^{1})\}.
%\end{aligned}
%\end{equation} 
\vspace{1 pc}

\begin{proof}

 Since $d\varphi+d^*\alpha\in L^{2,w}_{k}(i\Lambda^{1})$, it follows that $d^{+}d^{*}\alpha=\Delta_{+}\alpha\in L^{2,w}_{k-1}(i\Lambda^{2}_{+})$.
As proven above the maps 
\begin{align*}
\Delta_{+}{}_{\big|\{\mathbb{H}^{2}_{+}\}^{\perp_{L^{2, w}_{k+1}}}}:\{\mathbb{H}^{2}_{+}\}{}^{\perp_{L^{2, w}_{k+1}}}\ra \{\mathbb{H}^{2}_{+}\}^{\perp_{L^{2, w}_{k-1}}}
\end{align*}
are isomorphisms. Let $w>0$. Therefore in the following commutative diagram of Banach spaces
\begin{displaymath}
      \xymatrix{
      \{\mathbb{H}^{2}_{+}\}^{\perp_{L^{2,-w}_{k+1}}} \ar[r]^{ \Delta_{+}}& \{\mathbb{H}^{2}_{+}\}^{\perp_{L^{2,-w}_{k-1}}} \\
       \{\mathbb{H}^{2}_{+}\}^{\perp_{L^{2,+w}_{k+1}}} \ar@{^{(}->}[u] \ar[r]^{ \Delta_{+}} & \{\mathbb{H}^{2}_{+}\}^{\perp_{L^{2,+w}_{k-1}}}\ar@{^{(}->}[u]} \end{displaymath}
the horizontal arrows are isomorphisms (and the vertical arrows are continuous injections).
Let $\alpha\in L^{2,-w}_{k+1}(i\Lambda^{2}_{+})$ be such that $\Delta_{+}\alpha\in L^{2,w}_{k-1}(i\Lambda^{2}_{+})$.
Using the decomposition $\alpha=\alpha_{h}+\alpha^{\perp}$ where $\alpha_{h}\in \mathbb{H}^{2}_{+}$
and $\alpha^{\perp}\in\{\mathbb{H}^{2}_{+}\}^{\perp_{L^{2,-w}_{k+1}}}$, it follows, using the hypothesis, that $\Delta_{+}\alpha^{\perp}=\Delta_{+}\alpha\in\{\mathbb{H}^{2}_{+}\}^{\perp_{L^{2,+w}_{k-1}}}$.
Since the horizontal arrows are isomorphisms there is a $\beta\in \{\mathbb{H}^{2}_{+}\}^{\perp_{L^{2,+w}_{k+1}}}$ such that $\Delta_{+}\beta=\Delta_{+}\alpha^{\perp}$.
The upper arrows being injective we have necessarily, $\beta=\alpha^{\perp}$. Hence $\alpha=\alpha_{h}+\beta\in L^{2,w}_{k+1}(i\Lambda^{2}_{+}).$
Now from $d\varphi+d^*\alpha\in L^{2,w}_{k+1}(i\Lambda^{1})$ and $d^*\alpha\in L^{2,w}_{k+1}(i\Lambda^{1})$ it follows that $d\varphi\in L^{2,w}_{k+1}(\Lambda^{1})$. Applying Taubes Lemma  5.2 in \cite{Taubes} to each end $e\in\Eg$, it follows that $\varphi\in L^{2,w}_{k+1}(i\Lambda^{0})\oplus(\sum_{e\in\Eg} i\rho_{e}\R)=L^{2,w}_{k+1}(i\Lambda^{0})\oplus iF^\Eg$.
\end{proof}

\begin{coro}\label{Hodge}
Let $(X,g)$ be a Riemannian 4-manifold with periodic ends such that for every $e\in\Eg$ condition $(C)$ holds and $b^{+}_{2}(Y_e)=0$ . Suppose $w>0$ is a small weight for which $D_w$ is a Fredholm operator. Then
\begin{align*}
\{\ker D_w\}^{\perp_{ L^{2,w}_{k}}}&=\ima[ D^{*}_{-w}:L^{2,-w}_{k+1}(i\Lambda^{0}\oplus i\Lambda^{2}_{+})\ra L^{2,-w}_{k}(i\Lambda^{1})]\cap L^{2,w}_{k}(i\Lambda^{1})\\
&=\ima[ d+d^{*}:L^{2,-w}_{k+1}(i\Lambda^{0}\oplus i\Lambda^{2}_{+})\ra L^{2,-w}_{k}(i\Lambda^{1})]\cap L^{2,w}_{k}(i\Lambda^{1}).
\end{align*}
In particular we have $L^2$-orthogonal direct sum decompositions of topological vector spaces
\begin{align*}
L^{2,w}_{k}(i\Lambda^{1})&=\ker D_w\oplus \{\ker D_w\}^{\perp_{ L^{2,w}_{k}}}\\
&=\ker D_w\oplus\bigg\{ \ima[ d+d^{*}:L^{2,-w}_{k+1}(i\Lambda^{0}\oplus i\Lambda^{2}_{+})\ra L^{2,-w}_{k}(i\Lambda^{1})]\cap L^{2,w}_{k}(i\Lambda^{1})\bigg\}.
\end{align*}
and
\begin{align*}
L^{2,w}_{k}(i\Lambda^{1})&=\mathbb{H}^{1}_{w}\oplus
 \ima [d:L^{2,w}_{k+1}(i\Lambda^{0})\oplus i F^{\Eg}
\to L^{2,w}_{k}(i\Lambda^{1})]\oplus \\
&\oplus\ima [d^{*}:L^{2,w}_{k+1}(i\Lambda^{2}_{+})\ra L^{2,w}_{k}(i\Lambda^{1})].
\end{align*}
as topological vector spaces.

\end{coro}

\begin{proof}
Since $w$ is positive we have $\ker D_{w}\subseteq L^{2,-w}_{k}(i\Lambda^{1})$.  Therefore
$$\{\ker D_w\}^{\perp_{ L^{2,w}_{k}}}=\bigg\{\{\ker D_w\}^{\perp_{ L^{2,-w}_{k}}}\bigg\}\cap L^{2,w}_{k}(i\Lambda^1).
$$
By Lemma \ref{duality} we have
$$\{\ker D_w\}^{\perp_{ L^{2,-w}_{k}}}=\ima[ D^{*}_{-w}:L^{2,-w}_{k+1}(i\Lambda^{0}\oplus i\Lambda^{2}_{+})\ra L^{2,-w}_{k}(i\Lambda^{1})],
$$
hence the first equality follows.  For the second equality note  $\ker D_w$ is finitely dimensional,  and since $w>0$ the $L^2$ inner product on $L^{2,w}_{k}$ is continuous. This gives an $L^2$-orthogonal topological direct sum decomposition  
$$L^{2,w}_{k}(i\Lambda^{1})=\ker D_w\oplus \{\ker D_w\}^{\perp_{ L^{2,w}_{k}}},
$$
which proves the second equality. Using (\ref{avanthodge}) we get that 
\begin{align*}
&\{\ker D_w\}^{\perp_{ L^{2,-w}_{k}}}\cap L^{2,w}_{k}(i\Lambda^1)\\
&=\ima[ D^{*}_{-w}:L^{2,-w}_{k+1}(i\Lambda^{0}\oplus i\Lambda^{2}_{+})\ra L^{2,-w}_{k}(i\Lambda^{1})]\cap L^{2,w}_{k}(i\Lambda^1)\\
&=D^*_w\big\{\big(L^{2,w}_{k+1}(i\Lambda^{0})\oplus iF^\Eg\big)\oplus L^{2,w}_{k+1}(i\Lambda^{2}_{+})\big\}\\
&=d\big(L^{2,w}_{k+1}(i\Lambda^{0})\oplus iF^\Eg\big)+ d^*\big( L^{2,w}_{k+1}(i\Lambda^{2}_{+})\big).
\end{align*}
We know that the latter sum, say $S$, is closed. Now applying the operators $d$ and $d^*$, we see that the sum is direct, the first summand coincides $\ker d_{|S}$ and the second summand with  $\ker d^*_{|S}$. Therefore the two summands are closed, which completes the proof.

\end{proof}

\begin{coro}\label{isomor}
Let $(X,g)$ be a Riemannian manifold with periodic ends such that for every $e\in\Eg$ condition $(C)$ holds and $b^{+}_{2}(Y_e)=0$. Suppose that $w$ is a small positive weight for which the operator $D_w$ is Fredholm. Then the following map is an isomorphism of Banach spaces

\begin{align*}
d^{*}: {\mathbb{H}^{2}_{+} }^{\perp_{L^{2,w}_{k+1} (i\Lambda^{2}_{+})}}\ra \mathrm{Im }[d^{*}:L^{2,w}_{k+1}(i\Lambda^{2}_{+})\ra L^{2,w}_{k}(i\Lambda^{1})].
\end{align*}
\end{coro}

\begin{proof}
By the previous Lemma \ref{Hodge} we know that the right hand side is closed. It is also easy to see that $\ker[d^{*}:L^{2,w}_{k+1}(i\Lambda^{2}_{+})\ra L^{2,w}_{k}(i\Lambda^{1})]$ equals $\mathbb{H}^{2}_{+}$.

\end{proof}

%\begin{proof}
%Putting $K=\ker [d:L^{2,w}_{k}(i\Lambda^{1})\ra L^{2,w}_{k-1}(i\Lambda^{2}_{+})]$, $I=\ima [d:L^{2,k+1}(i\Lambda^{0})\oplus \rho\cdot i\R^{N}\ra L^{2,w}_{k}(i\Lambda^{1})])$ and $I^*=\ima [d^{*}: \textrm{Im } [\Delta_{+,w}:L^{2,w}_{k+3}(\Lambda^2_+)\ra L^{2,w}_{k+1}(\Lambda^2_+)]\ra L^{2,w}_{k}(i\Lambda^{1})]) $
% it follows from the above decomposition that
%\begin{align*}
%K&=(K\cap \mathscr{H}^{1})\oplus (K\cap I) \oplus (K\cap I^*) \\
%&=\mathscr{H}^{1}\oplus I.
%\end{align*}
%From this it follows surjectivity. It is easy to check that $L$ is injective.
%\end{proof}

We prove now that weights as in the hypothesis of Corollaries \ref{Hodge} and \ref{isomor}  do exist.  To do this we use Taubes result, Theorem \ref{Taubesgeneric}.
\begin{lemma}
Suppose that $(X,g)$ is a Riemannian manifold with periodic ends such that for every $e\in\Eg$ condition $(C)$ holds and $b^{+}_{2}(Y_e)=0$.  Then the operator $D_{w}=d^{+}\oplus d^{*}$ is Fredholm for all but a discrete subset of $w\in\R$.
\end{lemma}

\begin{proof}
We use Taubes's theorem for each $e\in\Eg$. If $e\in\Eg$, then from the hypothesis of $X$ it follows, $b^{0}(Y_e)-b^{1}(Y_e)+b^{2}_{+}(Y_e)=0$. Hence the index of the operator $(d^{+},d^{*}):A^{1}\ra A^2_+\oplus A^0$ on $Y_e$,  is zero. 
The action of $\gamma_e$ on $H^{1}(Y_e)$ is given by $\sigma(\gamma_e)(\xi)=[\gamma_e\wedge\xi]^{+}\oplus [\gamma_e^{*}\lrcorner \xi]$.
To check that it is injective we use the proof of Taubes (\cite{Taubes}, Lemma 3.2) that shows that a $[\xi]$ for which $[\gamma_e\wedge\xi]^{+}$ vanishes has to be a real multiple of $[\gamma_e]$. Now the vanishing of $[\gamma_e^{*}\lrcorner \xi]$ implies that $[\xi]$ has to be trivial. Hence we are in the conditions of Taubes' theorem and the result follows.
\end{proof}

An alternative proof is obtained using the computations made in the proof of Proposition \ref{dimb2} for $\xi\in S^1\setminus\{1\}$.
%
%
%We briefly note that for $w\geq 0$, $\mathbb{H}^{1}_{w}=\ker D_w=\ker d^{+}\oplus d^*=\ker d\cap\ker d^*=\ker \Delta^{(1)}_{w}$
%$\mathbb{H}^{+}_{w}=\ker d^*=\ker d^{+}d^*=\ker \Delta^{+}_{w}$

\section{Dirac operator}\label{diracsection}

Let $(X,g,\tau)$ be an  Riemannian 4-manifold endowed with a periodic $Spin^{c}(4)$ structure $\tau$.

We start by recalling the following classical 
\begin{thh}[Weitzenb\"ock]\label{weitzenbock}
Let $(X,g,\tau)$ be a Riemannian $Spin^c(4)$, 4-manifold. Suppose that $A\in \A(\det Q)$. Then for $\varphi\in\mathscr{C}_{c}^{\infty}(\Sigma^{+})$ we have
\begin{align*}
\sD_{A}\sD_{A}\varphi=\nabla_{A}^{*}\nabla_{A}\varphi+\frac{s_g}{2}\varphi+\rho(F_A^+)\varphi.
\end{align*}
\end{thh}

\begin{lemma}\label{twist}
Let $(X,g,\tau)$ be an oriented, Riemannian, $\textrm{Spin}^{c}(4)$, 4-manifold and $A\in \A(\det Q)$. Suppose a smooth map $f\in\mathscr{C}^{\infty}(X,\C)$ is given. Then $e^{-f}\sD_{A}(e^{f}\varphi)$ is given by
\begin{align*}
\sD_{A}\varphi+\Gamma(df)\cdot\varphi 
\end{align*}
for any $\varphi\in\mathscr{C}^{\infty}(X,\Sigma^{+})$.
\end{lemma}

\begin{proof}
Let $U$ be a sufficiently small open set of $X$, and $(e^i)$ be a local orthonormal frame on $U$. Then on $U$
\begin{align*}
&e^{-f}\sD_{A}(e^{f}\varphi)=e^{-f}\sum\Gamma(e^i)\cdot\nabla^{A}_{e_i}(e^{f}\varphi)\\
&=\sum\Gamma(e^i)\nabla^{A}_{e_i}\varphi+\sum\Gamma(e^i)\cdot df(e_i)\varphi=\sD_{A}\varphi+\Gamma(df)\varphi.
\end{align*}
\end{proof}

\begin{prop}\label{Dirac}
Let  $(X,g,\tau)$ be a Riemannian, $\mathrm{Spin}^{c}(4)$ 4-manifold with periodic ends such that  for each end $e\in\Eg$, $b_{1}(Y_e)=1$,  $s_{g_e}\geq 0$ and \linebreak
$\langle c_{1}(\det(Q_e))^2,L_e\rangle -\sigma(Y_e)=0$. Let $A\in\A(\det Q)$  be a connection whose restriction to each end $e\in\Eg$ coincides with the pullback of a connection $A_e\in\A(\det(Q_e))$ verifying $F_{A_e}^{+}=0$. Then the operator 
$\sD_{A}:\mathscr{C}_{c}^{\infty}(\Sigma^+)\ra\mathscr{C}_{c}^{\infty}(\Sigma^-)$ extends to a Fredholm operator on 
$\sD_{A}:L^2_{k+1}(\Sigma^+)\ra  L^2_{k}(\Sigma^-)$.
\end{prop}

\begin{proof}
We use Taubes' Lemma (\ref{Taubesexact}) for $w=0$. Let $\xi\in S^1$, and write it as $e^{i\lambda}$. To prove that the hypothesis of Taubes' criterion is satisfied we have to prove that for every end $e\in\Eg$, the conjugated operator $e^{-i\lambda\tau}\sD_{A_e}e^{i\lambda\tau}$  on $Y_e$ is an isomorphism (between the suitable Sobolev completions). According to Lemma \ref{twist} $e^{-i\lambda\tau}\sD_{A_e}e^{i\lambda\tau}= \sD_{A_{e}}+i\lambda\Gamma(d\tau)$. Since $\lambda\in \R$ and $d\tau$ is real, one can write $\sD_{A_{e}}+i\lambda\Gamma(d\tau)=\sD_{B^\lambda_e}$, where $B_e^\lambda \in\A(\det(Q_e))$ is the \textit{unitary} connection given by $B_e^{\lambda}=A+2i\lambda d\tau$.

On the other hand the index of $\sD_{A_{e}}+i\lambda\Gamma(d\tau)\varphi$ equals the index $\textrm{ind  }\sD_{A_{e}}=\frac{1}{4}(\langle c_{1}(L_e)^2,Y_e\rangle -\sigma(Y_e))$ which by hypothesis, is zero. Therefore to check that for all $\xi\in S^1$ the operators $e^{-i\lambda\tau}\sD_{A_e}e^{i\lambda\tau}$ are isomorphisms it is sufficient to check that their kernels are trivial. 
Note that $F_{B_e^\lambda}^+=F_{A_e}^+=0$. Thus Weitzenb\"ock formula (\ref{weitzenbock}) becomes
\begin{align*}
\sD_{B_e^\lambda}\sD_{B_e^\lambda}\varphi=\nabla_{B_e^\lambda}^{*}\nabla_{B_e^z}\varphi+\frac{s_{g_e}}{2}\varphi.
\end{align*}
Taking $L^2$ inner product with $\varphi$, it follows that
\begin{align*}
\|\sD_{B_e^\lambda}\varphi\|^2_{L^2}=\|\nabla_{A_e}\varphi\|^2_{L^2}+\int_{X} \frac{s_{g_e}}{2}|\varphi|^2.
\end{align*}
Since $s_{g_e}\geq 0$, the condition $\sD_{B_e^z}\varphi=0$ implies that  $\nabla_{\hat{A_e}}\varphi=0$. But if $\varphi$ was a nontrivial  parallel spinor, then $Y_e$ would be K\"ahler (see \cite{BLPR}, Theorem 1). In particular $b_1(Y_e)= 0$, which contradicts our assumptions. So $\varphi$ is trivial, and the result follows.
\end{proof}

\begin{re}
The same statement holds if one replaces the condition $b_{1}(Y_e)=1$ by the condition ``$Y_e$ does not admit any K\"abler structure".
\end{re}

\begin{coro}
There is $\varepsilon>0$ such that in the conditions of Proposition \ref{Dirac} the operator $\sD_{A}:L^{2,w}_{k+1}(\Sigma^+)\ra L^{2,w}_{k}(\Sigma^-)$ is Fredholm for every $w\in(-\varepsilon,\varepsilon)$.
\end{coro}

\chapter{The harmonic space $\mathbb{H}^1_w$ and the gauge group}

\section{The topological interpretation of $\mathbb{H}^1_w$}

Let $(X,g)$ is an Riemannian 4-manifold with periodic ends such that for every end $e\in\Eg$ condition $(C)$ holds and $b^{+}_{2}(Y_e)=0$ .
Then our Hodge decomposition Theorem, \ref{Hodge}, gives for any sufficiently small positive weight $w$
\begin{equation}\label{hodge1}
L^{2,w}_k(i\Lambda^1)= \mathbb{H}^1_w\oplus \textrm{Im }[d:L^{2,w}_{k+1}(i\Lambda^0)\oplus i F^\Eg \ra L^{2,w}_{k}(i\Lambda^1)]\oplus$$
$$\oplus\textrm{Im }[ d^*:L^{2,w}_{k+1}(i\Lambda^2_+)\ra L^{2,w}_{k}(i\Lambda^1)].
\end{equation}

This allows us to prove 

\begin{prop} \label{H1new} Let $(X,g)$ be a Riemannian 4-manifold with periodic ends such that for every end $e\in\Eg$ condition $(C)$ holds and $b^{+}_{2}(Y_e)=0$. Then the following holds
\begin{enumerate}
\item $ \mathbb{H}^1_w\subset \ker [d:L^{2,w}_k(i\Lambda^1)\ra L^{2,w}_{k-1}(i\Lambda^2)]$,
\item The inclusion $ \mathbb{H}^1_w\hookrightarrow \ker [d:L^{2,w}_k(i\Lambda^1)\to L^{2,w}_{k-1}(i\Lambda^2)]$ induces  isomorphisms
$$ \mathbb{H}^1_w\stackrel{\simeq}{\to}  \frac{\ker [d:L^{2,w}_k(i\Lambda^1)\rightarrow L^{2,w}_{k-1}(i\Lambda^2)]}{ \textrm{Im }[d:L^{2,w}_{k+1}(i\Lambda^0)\oplus i F^{\Eg}\ra L^{2,w}_{k}(i\Lambda^1)  ]}.
$$
\item The natural map 
$$c: \frac{\ker [d:L^{2,w}_k(i\Lambda^1)\to L^{2,w}_{k-1}(i\Lambda^2)]}{ \textrm{Im }[d:L^{2,w}_{k+1}(i\Lambda^0)\oplus iF^{\Eg}\ra L^{2,w}_{k}(i\Lambda^1)]}\to H^1_{\rm dR} (X,i\mathbb{R})
$$
is injective.
\item If for every $e\in\Eg$, the group $H_{1}(W_e,\Z)$ is torsion,  we have equalities 
$$\mathrm{Im}(c)=\mathrm{Im}[H^1_c(X,i\mathbb{R})\to H^1(X,i\mathbb{R})]=H^1(X,i\mathbb{R}),$$
hence natural isomorphisms
$$ \mathbb{H}^1_w\stackrel{\simeq}{\longrightarrow}\mathrm{Im }[H^1_c(X,i\mathbb{R})\to H^1(X,i\mathbb{R})]=H^1(X,i\mathbb{R})\ .
$$
\end{enumerate}
\end{prop}

\begin{proof} (1) For the first statement, recall the equality $\ker d^+=\ker d$ proved in the Appendix \ref{kerd}.
\\ \\
(2)  Applying $d$ to a  1-form $\alpha\in L^{2,w}_k(i\Lambda^1)$ decomposed according to (\ref{hodge1}) we obtain a direct sum decomposition
$$ \ker [d:L^{2,w}_k(i\Lambda^1)\to L^{2,w}_{k-1}(i\Lambda^2)]=\mathbb{H}^1_w\oplus\textrm{Im }[d:L^{2,w}_{k+1}(i\Lambda^0)\oplus  iF^{\Eg} \ra L^{2,w}_{k}(i\Lambda^1)],
$$
which proves the claim.
\\ \\
(3)  Let $\alpha\in L^{2,w}_k(i\Lambda^1)$ be a closed form such that $ [\alpha]_{\textrm{dR}}=0$.  Here we use implicitly the fact that the de Rham cohomology can be computed using Sobolev forms (or even distributions), hence a closed form in $L^{2,w}_k(i\Lambda^1)$ defines a de Rham cohomology class.

Therefore, there exists  a distribution $\phi\in \mathcal{D}'(X)$ on $X$ such that $d\phi=\alpha$. Using standard regularity theorems, we see that $\phi\in L^2_{k+1,loc}(X,i\mathbb{R})$.  By Lemma 5.2 p. 381 in \cite{Taubes}, it follows that $\phi\in L^{2,w}_{k+1}(i\Lambda^0)\oplus  iF^{\Eg}$, which proves the injectivity of $c$.
\\ \\
(4) Note first that that the inclusion 
$\mathrm{Im }[H^1_c(X,i\mathbb{R})\to H^1(X,i\mathbb{R})]\subset \mathrm{Im }(c) $ follows easily using de Rham theorem for cohomology with compact supports. This inclusion does not need the condition on $H_{1}(W_e,\Z)$, which is only needed for the opposite inclusion, as we will see below.

Since $H_{1}(W_e,\Z)$ is torsion,  we see by an iterated application of Mayer-Vietoris Theorem, for every end $e\in\Eg$ one has 
$$H_1(\cup_{0\leq i\leq k} W_{e_i},\Z) \hbox{ is torsion}
$$
Since homology commutes with inductive limits (see Theorem 4.1.5 p. 162 \cite{spanier}),  we obtain 
$$H_1(\End_e(X),\Z) \hbox{ is torsion}$$
 for every $e\in\Eg$. Using the universal coefficients theorem (see \cite{spanier} Theorem 3 p. 243) it follows $H^1(\End(X),\mathbb{R})=0$. Let $\alpha\in L^{2,w}_k(i\Lambda^1)$ be a closed form on $X$. Since  $H^1(\End(X),\mathbb{R})=0$, there exists   $\psi\in L^{2}_{k+1,loc}(\End(X))$ such that 
 $$d\psi=\alpha|_{\End(X)}.$$
   Let $\chi$ be a smooth cut-off function which is 1 on $\End_2(X)$ and 0 on $X\setminus \End_1(X)$ (see section \ref{periodicend} for notations).  Then  
$$c([\alpha])=[\alpha]_{\mathrm{dR}}=[\alpha -d (\chi \psi)]_{\mathrm{dR}},
$$
hence $c([\alpha])$ is represented by a form with compact support.

Finally, the equality  $\textrm{Im }[H^1_c(X,i\mathbb{R})\to H^1(X,i\mathbb{R})]=H^1(X,i\mathbb{R})$ can be proved in a similar way, using the same cut-off procedure applied this time to smooth closed 1-forms on $X$. 
\end{proof}

\begin{re}
The condition $H_1(W_e,\Z)$ is torsion is stronger than condition $(C)$ for the end $e\in\Eg$. It is an interesting question to give an explicit topological interpretation of $\mathbb{H}^{1}_{w}$ without this condition on the group $H_1(W_e,\Z)$.
\end{re}

\section{The gauge group}\label{gaugegroup}

Suppose $(X,g)$ is a manifold with periodic ends.
Consider for each end $e\in\Eg$, $\rho_{e}$ to be a smooth function on $X$ as in section \ref{hodgedeco}, which is zero outside $X_{e}^{\geq 0}$ and $1$ in $X_{e}^{\geq 1}$.

%Writing by $\theta_{\infty}:X\ra \C$  the function $\theta_{\infty}=\sum_{e\in\Eg} \rho_{e}\theta_{e}$, where $\theta_{e}\in \C$, it follows that  $\vartheta_{\infty}\xi_{\infty}=\sum_{e\in\Eg}\rho_{e}\theta_{e}\xi_{e}$. Thus we can identify  $\theta_{\infty}$ with $(\theta_e)_{e\in\Eg}\in \C^\Eg$, the above equality meaning simply that products commute with respect to this identification as do sums and inverses. 
%%$(\theta_e)_{e\in\Eg}\cdot(\xi_e)_{e\in\Eg}=(\theta_{e}\xi_{e})_{e\in\Eg}$.
%%If $\theta_{\infty}\in {(\C^{\ti})}^\Eg$ then its inverse equals, $\theta_{\infty}^{-1}=((\theta^{e})^{-1})_{e\in\Eg}$.

\vspace{1pc}

\begin{de}
Let $(X,g)$ be a Riemannian, $Spin^c(4)$, 4-manifold with periodic ends. Let $(\G_{w,k},\cdot)$ be the set
\begin{align*}
\G_{w,k}=\bigg\{&\theta \in L^{2}_{k+1,\mathrm{loc}}(X,\C): |\theta|=1 \textrm{ a.e}. \\
&\textrm{ such that  } \theta^{-1}d\theta \in L^{2,w}_{k}(\Lambda^{1}X)\bigg\} 
\end{align*}
endowed with the pointwise multiplication and inversion, which gives $(\G_{w,k},\cdot)$ the structure of an Abelian group.
\end{de}

Define now the morphism 
$$p:\mathscr{G}_{w,k+1}\to H^1(X,2\pi i\mathbb{ Z})
$$
by 
$$p(\theta):= [\theta^{-1} d\theta]_{\rm dR}\ .
$$
Using Cauchy formula it is easy to see that the right hand term defines an element of the image of $H^1(X,2\pi i\mathbb{Z})$ in $ H^1(X,i\mathbb{R})_{\mathrm{dR}}$. Using the definition of the gauge group $\mathscr{G}_{w,k+1}$ we obtain
\begin{equation}\label{kerp}
\ker p=\exp[ L^{2,w}_{k+1}(i\Lambda^0)\oplus F^{\Eg}],
\end{equation}
Indeed, supposing $\theta^{-1}d\theta=d\alpha$ for a smooth function $\alpha$, we get using Lemma 5.2 of \cite{Taubes} that $\alpha=\bar{\alpha}+\sum_{e\in\Eg}\rho_{e}\alpha_e$, for $\bar{\alpha}\in L^{2,w}_{k+1}$ and suitable constants $\alpha_e\in i\R$. A simple argument shows that $\theta$ is a (constant) multiple of $e^{\alpha}$.
Thus $\ker p$ coincides with the connected component of the gauge group $\mathscr{G}_{w,k+1}$. Under the condition $H_1(W_e,\Z)$ is torsion for every end $e\in\Eg$ (without any other condition on $X$) we can prove that $p$ is surjective. More precisely

\begin{prop} Let $(X,g)$ be a Riemannian 4-manifold with periodic ends, such that for every $e\in\Eg$ the group $H_1(W_e,\Z)$ is torsion.
\begin{enumerate}
\item One has a  short exact sequence
$$1\longrightarrow \exp[ L^{2,w}_{k+1}(i\Lambda^0)\oplus i F^{\Eg}]\hookrightarrow \mathscr{G}_{w,k+1}\stackrel{p}{\longrightarrow } 
H^1(X,2\pi i\mathbb{Z})\longrightarrow 0,$$
\item Putting
$$G:=\{\theta\in\mathscr{G}_{w,k+1}|\ \theta^{-1} d\theta\in \mathbb{H}^1_w\},$$
the restriction $p_G$ of $p$ to $G$ gives the short exact sequence
$$1\to S^1\to G\stackrel{p_G}{\longrightarrow } 
H^1(X, 2\pi i\mathbb{Z})\to 0.
$$

\end{enumerate}
\end{prop} 
\begin{proof} (1)   By (\ref{kerp})  we have $\ker(p)= \exp[ L^{2,w}_{k+1}(i\Lambda^0)\oplus i F^\Eg]$, hence it suffices to prove that $p$ is surjective.

Let $a\in H^1(X,2\pi i\mathbb{Z})$.  Let now $t\in (0,\infty)$ sufficiently large {\it regular value} of $\tau$ such that $X^{\geq t}\subset \End(X)$ (see section \ref{periodicend}).  Using again an iterated application of Mayer-Vietoris exact sequence, the functoriality of the homology functor with respect to inductive limits and the universal coefficients theorem as in  the proof of  Proposition \ref{H1new}, we obtain  $H^1(\End(X),\Z)=0$ . It follows that $a|_{X^{\geq t}}=0$. Therefore $a$ can be written as $2\pi i u(b)$, where 
$$u:H^1(X,X^{\geq t},\Z)\to H^1(X,\mathbb{Z})
$$
is the  natural map appearing in the cohomology long exact sequence associated with the pair $(X,X^{\geq t})$. The point is that this pair is a relative CW complex in the sense of \cite{spanier}. Since $S^1$ is a $K(\Z,1)$-space it follows (by Theorem 10 p. 428 \cite{spanier})  that there exists a   map  of pairs 
$$k:(X,X^{\geq t})\to (S^1,1)$$
such that $k^*(\{S^1\})=b$. This map can be regarded as a a continuous map $k:X\to S^1$ which is constant 1 on $X^{\geq t}$.  We can find a smooth approximation $f$ of $k$ which is constant 1 on $X^{\geq t+\epsilon}$ and is homotopically equivalent to $k$. We have obviously $f\in \mathscr{G}_{w,k+1}$ and  $a=p(f)$.

(2) Let $f\in G$ with $p(f)=0$. Using the first exact sequence, it follows that $f$ can be written as
$$f=\exp(\phi),
$$
where $\phi\in  L^{2,w}_{k+1}(i\Lambda^0)\oplus iF^\Eg$. But $f\in G$ implies
$$d\phi \in  \mathbb{H}^1_{w},
$$ 
which, by our Hodge decomposition Theorem (\ref{Hodge}) gives $d\phi=0$, i.e. $\phi\in i\mathbb{R}$, which gives $f\in S^1$.  It remains to prove that the restriction $p_G:=p|_G$ is still surjects onto $H^1(X ,2\pi i\mathbb{Z})$.  For a class  $a\in H^1(X,2\pi i\mathbb{Z})$ we obtain as above a smooth map $f:X\to S^1$ which is constant 1 on  $X^{\geq t+\epsilon}$ and such that $\frac{1}{2\pi i}a=f^*(\{S^1\})$. To complete the proof it suffices to find
$$\phi\in  L^{2,w}_{k+1}(i\Lambda^0)\oplus i F^\Eg
$$
such that $\exp(-\phi) f\in G$. This condition is equivalent to
$$ -d\phi+ \theta^{-1} d\theta\in \mathbb{H}^1.
$$
But $ \theta^{-1} d\theta \in \ker [d^+:L^{2,w}_{k}(i\Lambda^1)\to L^{2,w}_{k-1}(i\Lambda^2_+)]$, hence according to our Hodge decomposition theorem, it belongs to the sum
$$d( L^{2,w}_{k+1}(i\Lambda^{0})\oplus i F^{\Eg})\oplus \mathbb{H}^1 .
$$
 \end{proof}

The following lemma will play an important role in the construction of the Seiberg-Witten map as a map between Hilbert bundles over a torus (see section \ref{intro2}).
\def\Tors{\mathrm{Tors}}

\begin{lemma}\label{fingen}
Suppose that $H_1(W_e,\Z)$ is a torsion group for every $e\in\Eg$. Then  for a sufficiently large regular value of $\tau$, $a\in [0,\infty)$, the composition
$$H_1(X^{\leq a},\Z)\to H_1(X,\Z)\to H_1(X,\Z)/\Tors
$$
is surjective. In particular $H_1(X,\Z)/\Tors$ is finitely generated
\end{lemma}

\begin{proof}
As we have seen in the proof of Lemma 3.2.1 (4), using successively the Mayer-Vietoris exact sequence and the inductive limit property of the homology functor (see Theorem 4.1.5 p. 162 \cite{spanier}), we see that  the homology group $H_1(\End^e(X),\Z)$ is a torsion group for every $e\in\Eg$. The Mayer-Vietoris exact sequence for the pair
$(K,\End(X))$ contains the segment:
$$H_1(K,\Z)\oplus H_1(\End(X),\Z) \to H_1(X,\Z)\to H_0(N)\to H_0(K,\Z)\oplus H_0(\End(X),\Z),
$$ 
where $N:=\coprod_{e\in\Eg} N_e$. Note now, that any connected component of $N$ is contained in a connected component of $\End(X)$, hence the the right hand morphism is injective. This shows that the morphism

$$H_1(K,\Z)\oplus H_1(\End(X),\Z) \stackrel{a}{\to} H_1(X,\Z)
$$
given by $a(x,y)=(i_K)_*(x)+ (i_{\End(X)})_*(y)$ is surjective. Since
$H_1(\End(X),\Z)$ is a torsion group, it follows that the composition 
$$H_1(K,\Z)\stackrel{(i_K)_*}{\longrightarrow} H_1(X,\Z)\to  H_1(X,\Z)/\Tors
$$ 
is surjective.
Let now $a\in [0,\infty)$ be a sufficiently large regular value of $\tau$ such that $K\subset X^{\leq a}$. It suffices to note that the image of $H_1(K,\Z)$ in $H_1(X,\Z)$ is contained in the image of $H_1(X^{\leq a},\Z)$ in $H_1(X,\Z)$.

\end{proof}
\vspace{1mm}
{\ }
\begin{coro}\label{tensor}  Suppose that $H_1(W_e,\Z)$ is a torsion group for every $e\in\Eg$. Then
$$H^1(X,\Z)=\mathrm{Hom}(H_1(X,\Z)/\mathrm{Tors},\Z),\ H^1(X,\R)=\mathrm{Hom}(H_1(X,\Z)/\mathrm{Tors},\R),
$$
in particular, since $H_1(X,\Z)/\mathrm{Tors}$ is finitely generated by Lemma 3.2.3, it follows that $H^1(X,\R)=H^1(X,\Z)\otimes\R$.
\end{coro}
\begin{proof}
By the universal coefficients theorem one has
$$H^1(X,\Z)=\mathrm{Hom}(H_1(X,\Z),\Z)=\mathrm{Hom}(H_1(X,\Z)/\mathrm{Tors},\Z)
$$
and similarly for $H^1(X,\R)$. 
\end{proof}

Note that the universal coefficients Theorem relating cohomology to homology (see Theorem 5.5.3 p. 243 \cite{spanier}) does not assume $H_*(X,\Z)$ to be finitely generated. On the other hand the universal coefficients Theorem relating cohomology  to cohomology  (see Theorem 5.5.10 p. 246 \cite{spanier}) does need the assumption "$H_*(X,\Z)$ is finitely generated".

\chapter{Coercivity and compactness}

\section{Energy identities}
In this first section we let $X$ be a Riemannian 4-manifold with bounded geometry, i.e. complete and such that the infima of the injectivity radius is positive and all the derivatives of the Riemann tensor curvature are bounded in $\mathscr{C}^{k}$ for all $k$. This is true for a manifold with periodic ends and ensures density of compact support forms in the usual Sobolev spaces, as explained in \cite{A}.

\begin{thh}
Let $(X,g)$ be a Riemannian manifold with bounded geometry, then $\mathscr{C}^{\infty}_{c}$ is dense in $L^{p}_{k}$, and the Sobolev injection theorem hold.
\end{thh}
\begin{proof}
See Aubin's \cite{A},  2.7 and 2.21.
\end{proof}
We will refer to this theorem as Aubin-Sobolev.
As in $\cite{KM}$ let $(X,g,\tau)$ be a compact, oriented, Riemannian, $\mathrm{Spin}^{c}(4)$, 4-manifold with boundary. For a smooth pair $(A,\phi)\in\A(\det Q)\ti\mathscr{C}^{\infty}(\Sigma^+)$, one can define the Seiberg-Witten map 
\begin{align*}
SW(A,\varphi)=(\sD_A\varphi,\frac{1}{2}\rho(F_A^{+})-(\varphi\otimes\bar{\varphi})_{0})
\end{align*}
and the following topological energy and analytical energy integrals

\begin{align*}
\mathscr{E}^{\mathrm{top}}(A,\varphi)=\frac{1}{4}\int_{X}F_A\wedge F_A-\int_{\partial X}\langle \varphi,\sD_B\varphi \rangle+\int_{\partial X} \frac{H}{2}|\varphi|^2
\end{align*}

\begin{align*}
\mathscr{E}^{\mathrm{an}}(A,\varphi)=\frac{1}{4}\int_{X} |F_{A}|^2+\int_{X}|\nabla_{\hat{A}}\varphi|^2+\frac{1}{4}\int_{X}(|\varphi|^2+\frac{s_g}{2})^2-\int_{X}\frac{s_{g}^2}{16}
\end{align*}
where $B$ is the boundary $\textrm{Spin}^{c}$ connection and $H$ the mean curvature of the boundary. These verify the identity

\begin{align*}
\E^{\mathtt{an}}(A,\varphi)-\E^{\mathtt{top}}(A,\varphi)=\|SW(A,\varphi)\|^{2}_{L^{2}(X)},
\end{align*}
which plays an important role in the results obtained in \cite{KM}.

We will need a modified version of this identity which applies to pairs on non compact manifolds. 

Let  $(X,g,\tau)$ be an oriented Riemannian, $\textrm{Spin}^{c}(4)$, 4-manifold which we suppose connected.  Suppose $(A,\psi)\in\A(\det Q)\ti\mathscr{C}^{\infty}_{c}(\Sigma^+)$ is a pair with $F_{A}^{+}\in L^2$. Then $SW(A,\varphi)$ also belongs to $L^2$. Using the same arguments as in \cite{KM} (based on two applications of Stokes theorem)  we obtain the identity
\begin{align*}
\|SW(A,\varphi)\|^{2}_{L^2}=\frac{1}{2}\int_{X}|F_A^+|^2+\int_{X}|\nabla_{\hat{A}}\varphi|^2+\frac{1}{4}\int_{X}(|\varphi|^4+s_g|\varphi|^2).
\end{align*}
 note that for this identity to hold we don't need to assume that the integrals appearing in the definitions of $\E^{\mathrm{an}}$ and $\E^{\mathrm{top}}$ converge.
 
  If $A$ is connection that can be written as $A_0+v$ where $v$ is a smooth compactly supported imaginary 1-form, we have, recalling that for $\alpha\in i\Lambda^{k}$ we have $|\alpha|^{2}=-\alpha\wedge\ast\alpha$, 
\begin{align*}
-2|F_A^{+}|^2&=2F_A^+\wedge F_{A}^{+}\\
&= 2F_{A_0}\wedge F_{A_0}^{+}+2dv\wedge F_{A_0}^{+}+2F_{A_0}\wedge d^{+}v+dv\wedge dv-|dv|^{2}\\
&= 2F_{A_0}^{+}\wedge F_{A_0}^{+}+2dv\wedge\ast F^{+}_{A_0}+2F_{A_0}\wedge \ast d^{+}v+dv\wedge dv-|dv|^{2}\\
&= 2F_{A_0}^{+}\wedge F_{A_0}^{+}-2(dv, F^{+}_{A_0})-2(F_{A_0}, d^{+}v)+dv\wedge dv-|dv|^{2}\\
&= 2F_{A_0}^{+}\wedge F_{A_0}^{+}-4(dv, F^{+}_{A_0})+dv\wedge dv-|dv|^{2}\\
\end{align*}
pointwise. Thus the term$\frac{1}{2}\int_{X}|F_A^{+}|^2$ can be replaced by $\frac{1}{4}\int_{X}|dv|^{2}+4(dv, F^{+}_{A_0})-\frac{1}{2}\int_{X} F_{A_0}^{+}\wedge F_{A_0}^{+}$, where on the right all terms under the integral are compactly supported.

\begin{re}
Suppose $(X,g,\tau)$ is a Riemannian $\mathrm{Spin}^{c}(4)$, 4-manifold. Consider a connection $A_0\in\A(\det(Q))$ which is ASD at infinity,  let $A=A_0+v$ where $v\in \mathscr{C}_{c}(i\Lambda^{1})$, $\varphi\in\mathscr{C}_{c}(\Sigma^{+})$. Then putting
\begin{align*}
\widetilde{\E}^{\mathtt{an}}(A,\varphi)&:=\frac{1}{4}\int_{X}|dv|^{2}+4(F_{A_0}^{+},dv)+\int_{X}|\nabla_{\hat{A}} \varphi|^2+\frac{1}{4}\int_{X}|\varphi|^4+s_{g}|\varphi|^2,\\
\widetilde{\E}_{A_0}&=\frac{1}{2}\int_{X} F_{A_0}^{+}\wedge F_{A_0}^{+},
\end{align*}
we have
\begin{equation}\label{eq:newenergy} 
\begin{aligned} 
\widetilde{\E}^{\mathtt{an}}(A,\varphi)-\widetilde{\E}_{A_0}=\|SW(A,\varphi)\|^{2}_{L^{2}}.
\end{aligned}
\end{equation}
\end{re}
Our idea is to use this identity in a similar way as the energy identity of Kronheimer-Mrowka.

We now prove that  both these expressions and the identity (\ref{eq:newenergy}), extend to suitable spaces of Sobolev sections, provided that the underlying Riemannian manifold $(X,g)$ has bounded geometry.

\begin{prop}
\label{extension}
Let  $(X,g,\tau)$ be a  Riemannian 4-manifold with bounded geometry endowed with a $Spin^c(4)$ structure. Let $A_0\in\A(\det Q)$ be a smooth connection which is ASD outside a compact set. Then the above energy $\widetilde{\E}^{\mathtt{an}}$  extends naturally to  $[A_{0}+L^{2}_{1}(i\Lambda^{1})]\ti L^{2}_{1}(\Sigma^{+})$ as a continuous map.
\end{prop}

\begin{proof}
It suffices to note that $\widetilde{\E}^{\mathrm{an}}$ can be expanded as a sum of terms which can be extended continuously to $L^2_1$ using standard Sobolev embedding theorems on manifolds with bounded geometry. For instance the expansion of $\int_{X}|\nabla_{\hat{A}}\varphi|^2$ contains the term $\int_{X}\mathrm{Re}\langle \nabla_{A_0}\varphi,\gamma(v)\cdot\varphi\rangle$. On the other hand for the trilinear map $(v,\varphi,\psi)\ra \int_{X}\mathrm{Re}\langle \nabla_{A_0}\varphi,\gamma(v)\cdot\psi\rangle$ verifies the estimate 
\begin{align*}
| \int_{X}\mathrm{Re}\langle \nabla_{A_0}\varphi,\gamma(v)\cdot\psi\rangle|&\leq  \big\{\int_{X}|\nabla_{A_0}|^2\big\}^{1/2}\big\{\int_{X}|\gamma(v)\cdot\psi|^2\big\}^{1/2}\\
&\leq C\|\varphi\|_{L^2_1}\|v\|_{L^4}\|\psi\|_{L^4}\leq C'\|\varphi\|_{L^2_1}\|v\|_{L^2_1}\|\psi\|_{L^2_1},
\end{align*}
which shows that it extends continuously to $L^{2}_{1}\ti L^{2}_{1}\ti L^{2}_{1}$. Here we have used the bounded embedding $L^2_1\hookrightarrow L^4$ which, according to \cite{A} is valid on manifolds with bounded geometry. The other terms can be handled in a similar way.
\end{proof}

\begin{prop}\label{sw}
Suppose we are in the conditions of the Proposition \ref{extension}. Then the map $[A_{0}+\mathscr{C}_{c}(i\Lambda^{1})]\ti \mathscr{C}_{c}(\Sigma^{+})\ra\R$ given by $\|SW(A,\varphi)\|^{2}_{L^2}$, naturally extends to $[A_{0}+L^{2}_{1}(i\Lambda^{1})]\ti L^{2}_{1}(\Sigma^{+})$ as a continuous map.
\end{prop}

\begin{proof}
Remember that  $\|SW(A,\varphi)\|^{2}_{L^2}$ is given by
\begin{align}\label{swittennew}
\|\sD_{\hat{A}}\varphi\|^2_{L^2}+\|F^{+}_{A}-(\varphi\otimes\bar{\varphi})_{0}\|^2_{L^2}.
\end{align}
Now if $A=A+u$ then $\sD_{\hat{A}_0+u}\varphi=\sD_{\hat{A}_0}\varphi+\frac{1}{2}\Gamma(u)\varphi$. Since $\sD_{\hat{A_0}}\varphi$ is a contraction of $\nabla_{\hat{A_0}}\varphi$, it follows that also that $\|\sD_{\hat{A_0}}\varphi\|_{L^2}$ is bounded, up to a constant, by $\|\nabla_{\hat{A_0}}\varphi\|_{L^2}$. The expression $\Gamma(u)\varphi$ from the inclusion $L^{2}_{1}\hookrightarrow L^{4}$,  that the first term in $(\ref{swittennew})$ is a polynomial expression in $(u,\varphi)$ which is continuous with respect to the $L^{2}_{1}$.  The same argument applies to $F^{+}_{A}-(\varphi\otimes\bar{\varphi})_{0}=F_{A_0}+d^{+}v-(\varphi\otimes\bar{\varphi})_{0}$. The result follows from density smoothly compact supported sections as above.
\end{proof}

\begin{coro}[Energy Identity]\label{energyid}
In the conditions of the previous propositions the following equality holds
\begin{align*}
\widetilde{\E}^{\mathtt{an}}(A,\varphi)-\widetilde{\E}_{A_0}=\|SW(A,\varphi)\|^{2}_{L^2} 
\end{align*}o
for  all $(A,\varphi)\in[A_{0}+L^{2}_{1}(i\Lambda^{1})]\ti L^{2}_{1}(\Sigma^{+})$.
\end{coro}

\begin{proof}
It suffices to note that the result holds for smooth compactly supported sections.
\end{proof}

\section{The perturbed Seiberg-Witten map}

Fix a form  $\eta\in L^{2}_{k}(i\Lambda^{2}_{+})$, and consider the following perturbation of the Seiberg-Witten map
\begin{align*}
SW_{\eta}(v,\varphi)=(\sD_{A}\varphi, F_{A}^{+}+\eta-(\varphi\otimes\bar{\varphi})_{0}).
\end{align*}
Supposing  again that $(X,g,\tau)$ is an oriented, Riemannian, $Spin^{c}(4),$ 4-manifold  with bounded geometry, and that $A_0\in\A(\det Q)$ is a smooth connection which is $ASD$ outside a compact set, put
\begin{align*}
\widetilde{\E}_{\eta}^{\mathtt{an}}(A,\varphi)&:=\frac{1}{4}\int_{X}|dv|^{2}+2|\eta|^{2}+4(F_{A_0}^{+},dv)+4(F_{A_0}^{+},\eta)+4(\eta,dv)+\\
+&\int_{X}|\nabla_{\hat{A}} \varphi|^2+\frac{1}{4}\int_{X}|\varphi|^4+s_{g}|\varphi|^2-\frac{1}{2}\int_{X}\langle\varphi,\rho_{X}(\eta)\varphi\rangle,\\
\widetilde{\E}_{A_0, \eta} &=\frac{1}{2}\int_{X} F_{A_0}^{+}\wedge F_{A_0}^{+}.
\end{align*}
As before, these are  $\R$-valued continuous maps on $[A_0+L^{2}_{1}(i\Lambda^{1})]\ti L^{2}_{1}(\Sigma^{+})$. We have the following energy equality
\begin{align}\label{energyidpert}
\widetilde{\E}_{\eta}^{\mathtt{an}}(A,\varphi)-\widetilde{\E}_{A_0,\eta}=\|SW_{\eta}(v,\varphi)\|_{L^2}.
\end{align}
Note that, as before, the $\eta$-perturbed Seiberg-Witten map and energies are continuous with respect to $\eta$ in the $L^{2}_{1}(i\Lambda^{2}_{+})$ topology.

\section{A weak convergence lemma}\label{wconv}

From now on  we will suppose that $(X,g,\tau)$ is a Riemannian, $Spin^c(4)$, 4-manifold with periodic ends such that for every $e\in\Eg$, the following conditions hold

\begin{enumerate}
\item Condition $(C)$ stated in Proposition \ref{dimb2},
\item $ b^{2}_{+}(Y_e)=0$;
\item $s_{g_e}\geq 0$
\item $\langle c_{1}(L_e)^2,Y_e\rangle-\sigma(Y_e)=0$, and that
\end{enumerate}

Let $A_0\in\A(\det Q)$  is  a smooth connection whose restriction to each end $e\in\Eg$ coincides with the pullback of a connection $A_e\in\A(\det(Q_e))$ verifying $F_{A_e}^{+}=0$

We recall from sections \ref{hodgedeco} and \ref{diracsection} that, under these assumptions, the following holds:

\begin{enumerate}

\item There exists $\varepsilon>0$ such that 

\begin{enumerate} 
\item \label{orthog} For any $w\in(-\epsilon,\epsilon)$ the operator $\Delta_{+}:L^{2,w}_{k+3}(i\Lambda^{2}_{+})\ra L^{2,w}_{k+1}(i\Lambda^{2}_{+})$ is Fredholm, its kernel $\mathbb{H}^2_+$ is independent of $w$ and its image is the $L^2$-orthogonal complement to $\mathbb{H}^{2}_{+}$ in  $L^{2,w}_{k+1}(i\Lambda^{2}_{+})$,

\item For any $w\in(0,\epsilon)$ the operator $d^{*}:L^{2,w}_{k+1}(i\Lambda^{2}_{+})\ra L^{2,w}_{k}(i\Lambda^{1})$ has closed image, and the restriction of $d^{*}:L^{2,w}_{k+1}(i\Lambda^{2}_{+})\ra L^{2,w}_{k}(i\Lambda^{1})$ to the $L^2$-orthogonal complement of  $\mathbb{H}^{2}_{+}$ is an isomorphism onto this  image.

\end{enumerate}

\item There exists $\varepsilon>0$ such that the operator $\sD_{A_0}:L^{2,w}_{k}(\Sigma^+)\ra L^{2,w}_{k-1}(\Sigma^-)$ is Fredholm for all $w\in (-\epsilon,\epsilon)$.

\end{enumerate}

\begin{de}
A 4-tuple  $(X,g,\tau,A_0)$ where $(X,g,\tau)$ is a  $Spin^c(4)$ Riemannian 4-manifold with periodic ends, and $A_0\in\A(\det Q)$  satisfying the above conditions will be called \textit{admissible}.
\end{de}

In this section $(X,g,\tau,A_0)$ will always denote an admissible 4-tuple. We will consider pairs $(A,\varphi)$ where  $A$ has the form $A=A_0+v$ where $v\in \textrm{Im }[d^{*}:L^{2,w}_{k+2}(\Lambda^{2}_{+})\ra L^{2,w}_{k+1}(\Lambda^{1})]$, and where $\varphi\in L^{2,w}_{k}$. Recall from section \ref{hodgedeco} that for sufficiently small $w>0$ we have:  
\begin{re}\label{beta}
For any $v\in\textrm{Im }[d^{*}:L^{2,w}_{k+2}(\Lambda^{2}_{+})\ra L^{2,w}_{k+1}(\Lambda^{1})]$ there is a unique  $\beta\in {\mathbb{H}^{2}_{+}}^{\perp_{L^{2}}{L^{2,w}_{k+1}(\Lambda^{2}_{+})}}=\textrm{Im }[\Delta_{+}:L^{2,w}_{k+3}(i\Lambda^{2}_{+})\ra L^{2,w}_{k+1}(i\Lambda^{2}_{+})]$ with $d^{*}\beta=v$. 
\end{re}
This remark will allows us to make a variable change $v=d^*\beta$, which play an important role in our estimates. 

The goal of this chapter is to prove the following:

\begin{thh}
Let $(X,g,\tau,A_0)$ be an admissible 4-tuple.
Suppose that $w>0$ is sufficiently small and $k\geq 2$.
If
$(v_{n},\varphi_{n})_{n\in\mathbb{N}}$ is a sequence in  $\textrm{Im }[d^{*}:L^{2,w}_{k+1}(\Lambda^{2}_{+})\ra L^{2,w}_{k}(\Lambda^{1})]\ti L^{2,w}_{k}(\Sigma^{+})$
such that there exists $C>0$ with
\begin{equation}\label{bound}
\begin{aligned}
\|SW(v_{n},\varphi_{n})\|_{L^{2,w}_{k-1}}\leq C \ \forall n\in \mathbb{N}\ .
\end{aligned}
\end{equation}
There exists a subsequence of $ (v_{n},\varphi_{n})_{n\in\mathbb{N}}$  which is
\begin{enumerate}
\item   bounded in $L^{2,w}_{k}$,
\item    strongly convergent  in $L^{2,\bar{w}}_{k-1}$ for any  $0<\bar w<w$.\end{enumerate}
\end{thh}

\vspace{0.5 pc}
The proof uses ideas from  \cite{KM}  which will be adapted  for the new framework of manifolds with periodic ends.  
\vspace{1.5 pc}

We start with the following 
\begin{lemma}\label{wlemma}
Under the assumption, and with the notations above, let  $ (v_{n},\varphi_{n}) _{n\in\mathbb{N}}$ be a sequence in $\textrm{Im }[d^{*}:L^{2,w}_{k+1}(\Lambda^{2}_{+})\ra L^{2,w}_{k}(\Lambda^{1})]\ti L^{2,w}_{k}(\Sigma^{+})$, such that
\begin{align*}
\|SW(v_{n},\varphi_{n})\|_{L^{2,w}_{k-1}}\leq C.
\end{align*}
for a constant $C>0$. Then the sequence of quadruples $(v_{n},\varphi_{n},\nabla_{\hat{A}_{n}}\varphi_n,|\varphi_n|^2)_{n\in\mathbb{N},}$ is bounded in $L^2_{1}\ti L^2_1\ti L^2\ti L^2$, and thus it admits a subsequence which is  convergent weakly in $L^{2}_{1}\ti L^2_1\ti L^2\ti L^2$.
\end{lemma}

\begin{proof}
%Let $A_0$ be as before a fixed smooth connection ASD outside a compact set.
 %For $v_n\in \textrm{Im }[d^{*}:L^{2,w}_{k+1}(\Lambda^{2}_{+})\ra L^{2,w}_{k}(\Lambda^{1})]$,  write $(A_n,\varphi_n)=(A_0,0)+(v_n,\varphi_n)$.
Since  $w>0$ and $k\geq 1$ it follows from the weighted Sobolev embedding Theorem for manifolds with periodic ends   that  the inclusion $L^{2,w}_{k-1}\ra L^{2}$ is continuous.  Therefore the bound $(\ref{bound})$ implies an $L^2$-bound on  $(SW(v_{n},\varphi_{n}))_{n\in\Nat}$.

Provided $k\geq 1$ we have the energy identity
\begin{align*}
\widetilde{\E}^{\mathtt{an}}(A_{n},\varphi_{n})=\widetilde{\E}_{A_0}+\|SW(A_{n},\varphi_{n})\|^{2}_{L^2}
\end{align*}
where 
\begin{align*}
\widetilde{\E}^{\mathtt{an}}(A_{n},\varphi_{n})&=\frac{1}{4}\int_{X}|dv_n|^{2}-4(dv_n,F_{A_0}^{+})+\int_{X}|\nabla_{\hat{A}_n} \varphi_{n}|^2+\frac{1}{4}\int_{X}|\varphi_{n}|^4+s_{g}|\varphi_n|^2.
\end{align*}
From the bound $(\ref{bound})$ it follows that $\widetilde{\E}^{\mathtt{an}}(A_{n},\varphi_{n})$ is bounded too. The sum of the latter two terms will be written 
as $\frac{1}{4}(|\varphi_n|^{2}+s_{g}/2)^{2}-\frac{s_{g}^{2}}{16}$ on $X^{\leq 0}$ and as it stands on $X^{\geq 0}$. Therefore

\begin{align*}
\E^{\mathtt{an}}(A_{n},\varphi_{n})&=\frac{1}{4}\int_{X}|dv_n|^{2}+ \int_{X}|\nabla_{\hat{A}_{n}}\varphi_{n}|^2+\frac{1}{4}\int_{X^{\leq 0}} (|\varphi_n|^{2}+\frac{s_{g}}{2})^{2}-\int_{X^{\leq 0}}\frac{s_{g}^{2}}{16}\\
&+\int_{X^{\leq 0}}(-dv_{n},F_{A_0}^{+})+\frac{1}{4}\int_{X^{\geq 0}}|\varphi_{n}|^{4}+s_{g}|\varphi_{n}|^{2}\\
&\geq  \frac{1}{4}\int_{X}|dv_n|^{2}+\int_X |\nabla_{\hat{A}_{n}}\varphi_{n}|^2+\frac{1}{4}\int_{X^{\leq 0}} (|\varphi_n|^{2}+\frac{s_{g}}{2})^{2}-\int_{X^{\leq 0}}\frac{s_{g}^{2}}{16}\\
&-\int_{X}|(-dv_{n},F_{A_0}^{+})|+\frac{1}{4}\int_{X^{\geq 0}}|\varphi_{n}|^{4}+s_{g}|\varphi_{n}|^{2}.
\end{align*}

Considering $\epsilon>0$ but less than $\frac{1}{4}$ and using the inequality $|\langle a,b\rangle|\leq \varepsilon|a|^2+ C_{\varepsilon}|b|^2$,  we get

\begin{align*}
\E^{\mathtt{an}}(A_{n},\varphi_{n})&\geq \frac{1}{4}\int_{X}|dv_n|^{2}+ \int_{X}|\nabla_{\hat{A}_{n}}\varphi_{n}|^2+\frac{1}{4}\int_{X^{\leq 0}} (|\varphi_n|^{2}+\frac{s_{g}}{2})^{2}-\int_{X^{\leq 0}}\frac{s_{g}^{2}}{16}\\
&-\int_{X^{\leq 0}}(\varepsilon|dv_{n}|^2+C_{\varepsilon}|F_{A_0}^{+}|^{2})+\frac{1}{4}\int_{X^{\geq 0}}|\varphi_{n}|^{4}+s_{g}|\varphi_{n}|^{2}\\
\geq &(\frac{1}{4}-\varepsilon)\int_{X}|dv_n|^{2}+ \int_{X}|\nabla_{\hat{A}_{n}}\varphi_{n}|^2+\frac{1}{4}\int_{X^{\leq 0}} (|\varphi_n|^{2}+\frac{s_{g}}{2})^{2}-\int_{X^{\leq 0}}\frac{s_{g}^{2}}{16}\\
&-\int_{X}C_{\varepsilon}|F_{A_0}^{+}|^{2}+\frac{1}{4}\int_{X^{\geq 0}}|\varphi_{n}|^{4}+s_{g}|\varphi_{n}|^{2}.
\end{align*}

 The integrals of $s_{g}^{2}/16$ and $|F_{A_0}^{+}|^{2}$ are independent of $n$.  Since the scalar curvature is nonnegative on $X^{\geq 0}$, all other terms are non negative. Thus the bound on $\widetilde{\E}^{\mathtt{an}}$ allows us to get bounds on  $\|dv_n\|^{2}_{L^{2}(X)}$, $\|\nabla_{\hat{A}_n} \varphi_{n}\|^{2}_{L^2(X)}$,  $\int_{X^{\leq 0}} (|\varphi_n|^{2}+\frac{s_{g}}{2})^{2}$  and $\int_{X^{\geq 0}}|\varphi_{n}|^{4}$.
From  the latter two it follows using 
\begin{align*}
\||\varphi_n|^{2}\|_{L^2(X^{\leq 0})}\leq \||\varphi_n|^{2}+s_{g}/2)\|_{L^2(X^{\leq 0})}+\|s_{g}/2\|_{L^2(X^{\leq 0})},
\end{align*}
 that $\|\varphi\|_{L^4(X)}$ is bounded too.

Note that the bound on $\|dv_{n}\|^{2}_{L^2(X)}$ does not automatically control $\|v_n\|^{2}_{L^{2}(X)}$. To bound  $\|v_n\|^{2}_{L^{2}(X)}$ we will use  Remark \ref{beta} and the Fredholmness  $\Delta_{+}: L^{2}_{2}(\Lambda^{2}_+)\ra L^{2}(\Lambda^{2}_+) $  given by Proposition \ref{fredholmdelta}. Hence from
\begin{align*}
C\geq\|dv_{n}\|_{L^2}\geq\|d^{+}v_{n}\|_{L^2}&=\|d^{+}d^{*}\beta_n\|_{L^2}=\|\Delta_{+} \beta_n\|_{L^2}
\end{align*}
%
%where $v_n=d^{*}\beta_n$ for a unique %$\beta_n\in L^{2,w}_{k+2}(\Lambda^{2}_{+})\subseteq L^{2}$, 
we get 
\begin{align}\label{beta2}
\|\beta_n\|_{L^{2}_{2}(\Lambda^{2})}\leq \|\Delta_{+}\beta_{n}\|_{L^{2}(\Lambda^{2})}\|\leq C,
\end{align}
because $\beta_{n}\perp \ker_{L^{2}_{2}} \Delta_{+}$, by Remark $\ref{beta}$.
From the $L^2_2$-boundedness of $\beta_{n}$ we get the $L^{2}_{1}$-boundedness of $v_n=d^*\beta_n$.
\vspace{ 1 pc}

To complete the proof it remains to bound $(\varphi_{n})_n$ in $L^{2}_{1}$. 

Using  the $L^2$-bound on $SW(u_n,\varphi_n)$ we get an $L^2$-bound on its first component $\sD_{\hat{A}_n}\varphi_{n}$. 
The $L^{4}$-bound on $\varphi_n$ and the inclusion of $L^{2}_{1}\hookrightarrow L^{4}$  allow us to get the estimate
\begin{align*}
\|\frac{1}{2}\Gamma(v_{n})\varphi_{n}\|_{L^2}&\leq\|v_{n}\|_{L^4}\|\varphi_{n}\|_{L^4}\leq \|v_{n}\|_{L^{1}_{2}}\|\varphi_n\|_{L^4}.
\end{align*}
Thus we get an $L^2$-bound on
\begin{align*}
\sD_{\hat{A}_{0}}\varphi_{n}&=\sD_{\hat{A}_{n}}\varphi_{n}-\frac{1}{2}\Gamma(u_{n})\varphi_{n}.
\end{align*}
Now since  conditions (\ref{Dirac}) hold, the  operator $\sD_{A_{0}}$ is Fredholm and thus
\begin{align*}
\|\varphi_{n}\|_{L^{2}_{1}}\leq \|\sD_{A_0}\varphi_{n}\|_{L^2}+\|\varphi_{n}\|_{L^4}.
\end{align*}
It follows that $\varphi_{n}$ is bounded in $L^{2}_{1}$. This implies that $\varphi$ is bounded in $L^4$, hence $|\varphi|^2$ is bounded in $L^2$. In conclusion, the sequence 
$(v_n, \varphi_n,\nabla_{\hat{A}_{n}}\varphi_{n},|\varphi_n|^2)_{n}$ is bounded in $L^2_{1}\ti L^2_1\ti L^2\ti L^2$, hence it admits a weakly converging subsequence. 

\end{proof}

We note that the previous result (\ref{wlemma}) also holds for the $\eta$-perturbed Seiberg-Witten map.  The proof follows the same arguments, using Young's inequality to deal with some extra terms containing $\eta$.

\section{A $L^{1}_{loc}$ convergence lemma}

We know that $(v_n,\varphi_n,\nabla_{\hat{A}_{n}}\varphi_{n},|\varphi_n|^2)_n$ admits a subsequence  which converges weakly in $L^2_{1}\ti L^2_1\ti L^2\ti L^2$. We will denote by $(v_\infty,\varphi_{\infty},\lambda_{\infty},\mu_\infty)$ the weak $L^2_{1}\ti L^2_1\ti L^2\ti L^2$-limit of this subsequence. In order to save on notations we will still denote the obtained subsequence by  $(v_n,\varphi_n,\nabla_{\hat{A}_{n}}\varphi_{n},|\varphi_n|^2)_n$.

\begin{lemma}
Let $(v_n,\varphi_n$) be the above subsequence. Then 
\begin{equation}\label{L1loc}
\begin{aligned}
 &SW(v_n,\varphi_n)\ra SW(v_\infty,\varphi_\infty) \hbox{, }\nabla_{A_n}\varphi_n\ra \nabla_{A_\infty}\varphi_\infty \hbox{, }  |\varphi_n|^2\ra |\varphi_\infty|^2
\end{aligned}
\end{equation}
in $L^{1}_{loc}$.
\end{lemma}

\begin{proof}
Let $U$ be an open subset of $X$  with smooth boundary and compact closure and the restriction of $(v_n,\varphi_n)$ to $U$ by $(v_n,\varphi_n)|_U$. By Lemma (\ref{wlemma})  $(v_n,\varphi_n)$ is bounded in $L^2_1$ and converges weakly in $L^2_1$ to $(v_\infty,\varphi_{\infty})$. Thus $(v_n,\varphi_n)|_U$ is  bounded in $L^2_1(U)$  and  converges weakly in $L^{2}_1(U)$ to $(v_\infty,\varphi_\infty)|_U$ in $L^{2}_{1}(U)$.
Using Rellich-Kondrashov for compact manifolds with boundary (see \cite{A} 2.25), then, passing to a subsequence if necessary, $(v_n,\varphi_n)|_U$ converges strongly in $L^{p}(U)$, for $1\leq p <4$.

Note that, as in (\ref{sw}), the Seiberg-Witten map can be written as \begin{align*}
SW(v,\varphi)=C_{A_0}+D(v,\varphi)+B(v,\varphi).
\end{align*}
Here $C_{A_0}$ is a constant term, $D$ a is continuous linear map $L^{2}_{1}(U)\ra L^2(U)$ $B:L^2_{1}(U)\ra L^{2}(U)$ is a quadratic polynomial, continuous with respect to the  $L^2_{1}(U)$ topology.
We want to prove that
\begin{align*}
SW(v_n,\varphi_n)\ra SW(v_\infty,\varphi_\infty)
\end{align*}
in $L^{1}_{loc}(U)$. Since $D$ is linear and continuous, it is $L^2_1$-weakly continuous. Thus we only need to focus on the quadratic part.  $B(v,\varphi)$ has two components $\Gamma(v)\varphi$ and $(\varphi\otimes\bar{\varphi})_{0}$. We first prove that  $\Gamma(v_n)\varphi_n\ra \Gamma(v_\infty)\varphi_\infty$ in $L^{1}_{loc}$. Using the decomposition
\begin{align*}
\Gamma(v_n)\varphi_n-\Gamma(v_\infty)\varphi_\infty=\Gamma(v_n)(\varphi_n-\varphi_\infty)+\Gamma(v_n-v_\infty)\varphi_\infty,
\end{align*}
it follows, using the $L^{4}(U)$-bounds on $v_n$ and $\varphi_\infty$, together with the  $L^{4-\varepsilon}(U)$ strong convergence of $(v_n,\varphi_n)|_U$ that $\Gamma(v_n)\varphi_n\ra \Gamma(v_\infty)\varphi_\infty$ for $L^{2-\varepsilon'}(U)$. A similar argument holds for both the term $(\varphi_n\otimes\bar{\varphi}_n)_{0}$. Therefore the $L^1_{loc}$ limit of $SW(v_n,\varphi_n)$ is $SW(v_\infty,\varphi_\infty)$. To prove that $\nabla_{A_n}\varphi_n$ converges in $L^1_{loc}$ to $\nabla_{A_\infty}\varphi_\infty$ we use the development $\nabla_{A_n}\varphi_n=\nabla_{A_0}\varphi_n+\frac{1}{2}\Gamma(v_n)\varphi_{n}$ and the same method.  For proving that, it suffices to recall that $|\varphi_n|^2\ra |\varphi_\infty|^2$ in $L^1_{loc}$, note that $\varphi_n\ra \varphi_\infty$ strongly in $L^{4-\epsilon}(U)$.  Thus $|\varphi_n|\ra |\varphi_\infty|$ converges strongly in $L^{4-\epsilon}(U)$.
  
%Thus for any test function pair $\Psi$ with support in $U$ we have
%\begin{align*}
% \int_U SW(v_n,\varphi_n)\Psi \ra\int_U SW(v_\infty,\varphi_\infty)\Psi
%\end{align*}
%
\end{proof}
Note that the above results extends to $SW_{\eta}$ too, using an analogous decomposition of $SW_{\eta}$ as a sum of a constant plus a linear plus a quadratic term.

\begin{coro}
Let $(\lambda_{\infty},\mu_\infty)$ be the weak $L^2\ti L^2$-limit of $(\nabla_{\hat{A}_{n}}\varphi_{n},|\varphi_n|^2)_n$. Then $(\lambda_{\infty},\mu_\infty)=(\nabla_{A_\infty}\varphi_\infty,|\varphi_\infty|^2)$.
\end{coro}

\section{A strong convergence lemma}

In this section we will prove that the weakly converging subsequence obtained in the Section (\ref{wconv}) has a strongly convergent subsequence.
 We will make use of a  fundamental result in Functional Analysis. 

Recall first that, for a  Banach $E$ and and  a  sequence $(x_n)_{n}$ in $E$ which is weakly convergent to an element $x\in E$ we have 
\begin{equation}\label{liminf}
\liminf \|x_n\| \geq \|x\|. 
\end{equation}
If now  $E$ is a Hilbert space and $(x_n)_{n}$ is a  sequence in $E$ which is weakly convergent to an element $x\in E$, then the strong convergence, $x_n\ra x$, is equivalent to the norm convergence,  $\|x_n\|\ra\|x\|$. A fundamental result of Radon-Riesz states that the same property holds for the $L^p$-spaces (see \cite{radon} Corollary 5.2.19, p. 454):

\begin{lemma}\label{radonriesz}
Let $\mu$ be a positive measure on a $\sigma$-algebra $\Sigma$ of subsets of a set $\Omega$ and  $1<p<\infty$.
Let  $(x_n)_{n}$ is  a sequence in $L^p(\Omega,\Sigma,\mu)$ which converges weakly to $x\in  L^p(\Omega,\Sigma,\mu)$ and $\|x_n\|_{L^p}\to  \|x\|_{L^p}$, then $(x_n)_{n}$ converges strongly to $x$ in  $L^p(\Omega,\Sigma,\mu)$ 
\end{lemma}

We are now ready to prove that

\begin{lemma}\label{slemma}
The above weakly converging subsequence $(v_n,\varphi_n)_n$ has a subsequence converging strongly in $L^{2}_{1}$ to $(v_\infty,\varphi_\infty)$.
\end{lemma}

\begin{proof}
Since  $w>0$ and $k\geq 2$, the inclusion of weighted Sobolev spaces,  $L^{2,w}_{k-1}\hookrightarrow L^2$ is compact. Using the $L^{2,w}_{k-1}$-bound on $SW(v_n,\varphi_n)$  (see (\ref{bound})) it follows that, up to a subsequence, $SW(v_n,\varphi_n)$ converges strongly in $L^2$ to a limit, $(\Phi,\omega)\in L^2$. On the other hand, we know from the previous section that  $SW(v_n,\varphi_n)$ converges in $L^1_{loc}$ to $SW(v_\infty,\varphi_\infty)$. Thus $(\Phi,\omega)=SW(v_\infty,\varphi_\infty)$ (because they coincide as distributions) and $SW(v_n,\varphi_n)$ converges strongly to $SW(v_\infty,\varphi_\infty)$ in $L^2$.

Using  the energy identity (\ref{energyid}) it follows that
\begin{align*}
\widetilde{\E}^{\mathtt{an}}(v_{n},\varphi_{n})-\widetilde{\E}_{A_0}=\|SW(v_{n},\varphi_{n})\|^{2}_{L^2}\ra\|SW(v_{\infty},\varphi_{\infty})\|^{2}_{L^2}=\widetilde{\E}^{\mathtt{an}}(v_{\infty},\varphi_{\infty})-\widetilde{\E}_{A_0},
\end{align*}
and thus, 
\begin{align}\label{lim}
\lim_{n\to\infty} \widetilde{\E}^{\mathtt{an}}(v_{n},\varphi_{n})=\widetilde{\E}^{\mathtt{an}}(v_{\infty},\varphi_{\infty}).
\end{align}
 From this we prove strong $L^2_1$-convergence of $(v_n,\varphi_n)$. To do this rewrite the defining formula for  $\widetilde{\E}^{\mathtt{an}}(v_{n}\varphi_{n})$ as
\begin{equation}\label{int}
\begin{aligned}
&\frac{1}{4}\|dv_{n}\|^{2}_{L^{2}(X)}+\frac{1}{4}\|\varphi_n\|^{4}_{L^{4}(X^{>0})}+\int_{X^{>0}}s_{g}|\varphi_n|^2+\|\nabla_{\hat{A}_n} \varphi_{n}\|^{2}_{L^{2}(X)}\\
&+\frac{1}{4}\||\varphi_{n}|^2+s_{g}/2\|^{2}_{L^{2}(X^{\leq 0})}=\widetilde{\E}^{\mathtt{an}}(v_{n},\varphi_{n})+\langle dv_n,F_{A_0}^{+}\rangle+\int_{X^{\leq 0}}\frac{s^{2}_{g}}{16}
\end{aligned}
\end{equation}
%
%as a sum of positive terms plus $L^2_1$-weakly continuous terms. \textit{note that, here, we assume that the scalar curvature is nonnegative on $X^{\geq 0}$}.
Using (\ref{lim}) and the weak $L^2_1$-convergence of $(v_{n},\varphi_{n})$ it follows that the right hand side of (\ref{int}) converges to 
\begin{equation}
\begin{aligned}
&\widetilde{\E}^{\mathtt{an}}(v_{\infty},\varphi_{\infty})+\langle dv_\infty,F_{A_0}^{+}\rangle+\int_{X^{\leq 0}}\frac{s^{2}_{g}}{16}=\frac{1}{4}\|dv_{\infty}\|^{2}_{L^{2}(X)}+\frac{1}{4}\|\varphi_\infty\|^{4}_{L^{4}(X^{>0})}\\
&+\int_{X^{>0}}s_{g}|\varphi_\infty|^2+\|\nabla_{\hat{A}_\infty} \varphi_{\infty}\|^{2}_{L^{2}(X)}
+\frac{1}{4}\||\varphi_{\infty}|^2+s_{g}/2\|^{2}_{L^{2}(X^{\leq 0})}.
\end{aligned}
\end{equation}

Therefore

\begin{equation}\label{limnorm}
\begin{aligned}
\lim_{n\to\infty} \bigg(\frac{1}{4}\|dv_{n}\|^{2}_{L^{2}(X)}+\frac{1}{4}\|\varphi_n\|^{4}_{L^{4}(X^{>0})}+\int_{X^{>0}}s_{g}|\varphi_n|^2+\|\nabla_{\hat{A}_n} \varphi_{n}\|^{2}_{L^{2}(X)}\\
+\frac{1}{4}\||\varphi_{n}|^2+s_{g}/2\|^{2}_{L^{2}(X^{\leq 0})}\bigg)=\frac{1}{4}\|dv_{\infty}\|^{2}_{L^{2}(X)}+\frac{1}{4}\|\varphi_\infty\|^{4}_{L^{4}(X^{>0})},\\
+\int_{X^{>0}}s_{g}|\varphi_\infty|^2+\|\nabla_{\hat{A}_\infty} \varphi_{\infty}\|^{2}_{L^{2}(X)}
+\frac{1}{4}\||\varphi_{\infty}|^2+s_{g}/2\|^{2}_{L^{2}(X^{\leq 0})}.
\end{aligned}
\end{equation}

Recalling that we assumed $s_g\geq 0$ on $X^{>0}$ and given the bounds of Lemma (\ref{wlemma}), it follows that the sequence formed by the five terms on the left belongs to a compact parallelepiped in $\R_{\geq 0}^5$. Hence we can find a subsequence $(n_k)_k$ of $\Nat$ such that the five sequences 
\begin{equation}
\begin{aligned}
 \big(\|dv_{n_k}\|^{2}_{L^{2}(X)})_k \big)\ ,\ \big( \frac{1}{4}\|\varphi_{n_k}\|^{4}_{L^{4}(X^{>0})}\big)_k,
 \big(\|\sqrt{s_{g}}\varphi_{n_k}\|^2_{L^2(X^{>0})}\big)_k\\ 
 \big(\|\nabla_{\hat{A}_{n_k}} \varphi_{n_k}\|^{2}_{L^{2}(X)} \big)_k, \big(\frac{1}{4}\||\varphi_{n_k}|^2+s_{g}/2\|^{2}_{L^{2}(X^{\leq 0})}\big)_k
\end{aligned}
\end{equation}
converge simultaneously.

On the other hand, using (\ref{liminf}) and the weak convergences given by Lemma (\ref{wlemma}), we obtain 
\begin{equation}
\begin{aligned}
&\lim \big(\|dv_{n_k}\|^{2}_{L^{2}(X)})_k \big)\geq \|dv_{\infty}\|^{2}_{L^{2}(X)},\ \lim\big( \frac{1}{4}\|\varphi_{n_k}\|^{4}_{L^{4}(X^{>0})}\big)_k\geq  \frac{1}{4}\|\varphi_{\infty}\|^{4}_{L^{4}(X^{>0})},\\
&\lim\big(\|\sqrt{s_{g}}\varphi_{n_k}\|^2_{L^2(X^{>0})}\big)_k\geq \|\sqrt{s_{g}}\varphi_{\infty}\|^2_{L^2(X^{>0})},\\
&\lim \big(\|\nabla_{\hat{A}_{n_k}} \varphi_{n_k}\|^{2}_{L^{2}(X)} \big)_k\geq\|\nabla_{\hat{A}_{\infty}} \varphi_{\infty}\|^{2}_{L^{2}(X)},\\
 &\lim\big(\frac{1}{4}\||\varphi_{n_k}|^2+s_{g}/2\|^{2}_{L^{2}(X^{\leq 0})}\big)_k\geq \frac{1}{4}\||\varphi_{\infty}|^2+s_{g}/2\|^{2}_{L^{2}(X^{\leq 0})}.
\end{aligned}
\end{equation}
Combining with (\ref{limnorm}), we see that all the inequalities above are equalities. Therefore by the Radon-Riesz theorem (\ref{radonriesz}) we obtain strong convergences, 
\begin{equation}
\begin{aligned}
&dv_{n_k}\ra_{L^{2}(X)} dv_\infty , \ \varphi_n\ra_{L^{4}(X^{>0})}\varphi_\infty ,\ \nabla_{\hat{A}_n} \varphi_{n},
\ra_{L^{2}(X)}\nabla_{\hat{A}_\infty} \varphi_{\infty}, \\
& |\varphi_{n_k}|^2+s_{g}/2\ra_{L^{2}(X^{\leq 0})}|\varphi_{\infty}|^2+s_{g}/2
\end{aligned}
\end{equation}
Using Remark \ref{beta} and the $L^2$-convergence of $dv_{n_k}\ra dv_\infty $, it follows that 
\begin{align*}
\|\Delta_{+}( \beta_{n_k}-\beta_{\infty})\|_{L^2}&=\|d^{+}d^{*}(\beta_{n_k}-\beta_{\infty})\|_{L^2}=\|d^{+}(v_{n_k}-v_{\infty})\|_{L^2}\ra 0.
\end{align*}
This together with the estimate
\begin{align*}
\|\beta_{n_k}-\beta_{\infty}\|_{L^{2}_{2}}&\leq \|\Delta_{+}(\beta_{n_k}-\beta_{\infty})\|_{L^{2}}+\|\pi_{\ker \Delta_{+}}(\beta_{n_k}-\beta_{\infty})\|_{L^2}
\end{align*}
proves as in (\ref{beta2}), shows that $v_n\ra v_\infty$ strongly in $L^{2}_{1}$.  

Boundedness of  $s_{g}$ in $X^{\leq 0}$, and  the  convergence of $|\varphi_{n_k}|^2+s_{g}/2\ra|\varphi_{\infty}|^2+s_{g}/2$ in $L^{2}(X^{\leq 0})$, imply that $|\varphi_{n_k}|^2\ra |\varphi_{\infty}|^2$ in $L^{2}(X^{\leq 0})$. This with the $L^4(X^{\geq 0})$-convergence of $\varphi_{n_k}\ra\varphi_{\infty}$ , proves that $\varphi_{n_k}\ra \varphi_{\infty}$ in $L^{4}(X)$. 

From the $L^2$-convergence of $\nabla_{\hat{A}_{n_k}} \varphi_{n}\ra\nabla_{\hat{A}_\infty} \varphi_{\infty}$,  it follows that
\begin{align*}
\nabla_{\hat{A}_0}\varphi_{n_k}+\frac{1}{2}\Gamma(u_{n_k})\varphi_{n_k}\ra\nabla_{\hat{A}_0}\varphi_\infty+\frac{1}{2}\Gamma(u_\infty)\varphi_\infty
\end{align*}
in $L^2$. Therefore to prove that $\nabla_{\hat{A}_0}\varphi_{n_k}\ra\nabla_{\hat{A}_0}\varphi_\infty$ in $L^2$ it is enough to prove that $\Gamma(v_{n_k})\varphi_{n_k}\ra\Gamma(v_\infty)\varphi_\infty$ in $L^2$. This follows from the strong convergence of $_{n_k}\ra v_\infty$ in $L^{2}_1$ together with Aubin-Sobolev $L^{2}_{1}\hookrightarrow L^{4}$, and the strong convergence $\varphi_{n_k}\ra\varphi_\infty$ in $L^4$. Since that $\sD_{\hat{A}_0}$ is Fredholm and  is a contraction of $\nabla_{\hat{A}_0}$ we have
\begin{align*}
\|\varphi_{n_k}-\varphi_\infty \|_{L^{2}_{1}}&\leq C\|\sD_{\hat{A}_0}(\varphi_{n_k}-\varphi_\infty)\|_{L^{2}}+ C\|\varphi_{n_k}-\varphi_\infty\|_{L^4}\\
&\leq C\|\nabla_{\hat{A}_0}(\varphi_{n_k}-\varphi_\infty)\|_{L^{2}}+C\|\varphi_n-\varphi_\infty\|_{L^4}
\end{align*}
And thus $\varphi_{n_k}\ra \varphi_\infty$ in $L^{2}_{1}$.
\end{proof}

\section{Coercivity}

We arrive to the main theorem of this section 

\begin{thh}\label{coerc}
Let $(X,g,\tau,A_0)$ be an admissible 4-tuple.
Suppose that $w>0$ and $k\geq 2$.
Let $(v_{n},\varphi_{n})_{n\in\mathbb{N}}$ be a sequence in  
$$\textrm{Im }[d^{*}:L^{2,w}_{k+1}(\Lambda^{2}_{+})\ra L^{2,w}_{k}(\Lambda^{1})]\ti L^{2,w}_{k}(\Sigma^{+})$$
such that there exists $C>0$ with
\begin{equation}\label{bound}
\begin{aligned}
\|SW(v_{n},\varphi_{n})\|_{L^{2,w}_{k-1}}\leq C \ \forall n\in \mathbb{N}\ .
\end{aligned}
\end{equation}

\begin{enumerate}
\item   There exists a subsequence of $ (v_{n},\varphi_{n})_{n\in\mathbb{N}}$  which is bounded in $L^{2,w}_{k}$,
\item    For any $\lambda\in (0, w)$ there exists a subsequence of $ (v_{n},\varphi_{n})_{n\in\mathbb{N}}$  which is strongly convergent  in $L^{2,\lambda}_{k-1}$. \end{enumerate}
\end{thh}

Note:  In the  following proof all irrelevant constants will be denoted by $C$.

\begin{proof} {\ }\\
{\it Step 1.}    For any $\lambda\in (0, w)$ the sequence  $ (v_{n},\varphi_{n})_{n\in\mathbb{N}}$ has a subsequence which converges strongly in $L^{3,\lambda}_1$.
\\
\\
We start noting that $SW(v,\varphi)$ can be written  as
\begin{align}\label{SW}
SW(v,\varphi)=C_{A_0}+D(v,\varphi)+(v,\varphi)\sharp(v,\varphi).
\end{align}
where $C_{A_0}$ is a constant,   $D=(d^{+},D_{A_{0}})$ is a Fredholm operator  (for sufficiently small weights $w$) and $\sharp$ is  a continuous symmetric bilinear operator. 

According to  (\ref{slemma}) the sequence $(v_{n},\varphi_{n})_n$ has a subsequence $(v_{n_i},\varphi_{n_i})_i$ which is strongly convergent  in $L^2_1$. Put $(\gamma_{i})_i:=(v_{n_i},\varphi_{n_i})_i$ to save on notations, and let   $i_0\in\Nat$. It follows from (\ref{SW}) that
\begin{equation}\label{Fredholm}
\begin{aligned}
D(\gamma_{i}-\gamma_{j})&=SW(\gamma_{i})-SW(\gamma_{j})+\gamma_{j}\sharp\gamma_{j}-\gamma_{i}\sharp\gamma_{i}=SW(\gamma_{i})-SW(\gamma_{j})\\
&+(\gamma_{j}-\gamma_{i})\sharp(\gamma_{j}+\gamma_{i}-2\gamma_{i_{{0}}})+(\gamma_{j}-\gamma_{i})\sharp 2\gamma_{i_{0}}.
\end{aligned}
\end{equation}
Let $\lambda \in(0,w)$. We prove that $(\gamma_i)_i$ has a converging subsequence in $L^{3,\lambda}_1$. Given   $\epsilon>0$ choose $i_{\epsilon}\in\Nat$ such that for all $i\geq i_{\epsilon}$ the inequality $\|\gamma_{i}-\gamma_{j}\|_{L^{2}_{1}}\leq \epsilon$ holds.
Since $D$ is Fredholm  (for sufficiently small weights) there is a constant $C>0$, such that 
\begin{align}\label{ineqcoerciv1}
\|\gamma_{i}-\gamma_{j}\|_{L^{3,\lambda}_{1}}\leq C(\|D(\gamma_{i}-\gamma_{j})\|_{L^{3,\lambda}}+\|\gamma_{i}-\gamma_{j}\|_{L^2_1}).
\end{align}
Hence, using (\ref{Fredholm}), we have
\begin{equation}\label{ineqcoerciv2}
\begin{aligned}
\|\gamma_{i}-\gamma_{j}\|_{L^{3,\lambda}_{1}}&\leq C(\|D(\gamma_{i}-\gamma_{j})\|_{L^{3,\lambda}}+\|\gamma_{i}-\gamma_{j}\|_{L^2_1})\\
&\leq C\|SW(\gamma_{i})-SW(\gamma_{j})\|_{L^{3,\lambda}}+C\|(\gamma_{j}-\gamma_{i})\sharp(\gamma_{j}+\gamma_{i}-2\gamma_{i_{\epsilon}})\|_{L^{3}}+\\
&  +C\|(\gamma_{j}-\gamma_{i})\sharp 2\gamma_{i_{\epsilon}}\|_{L^{3,\lambda}}+C\|\gamma_{i}-\gamma_{j}\|_{L^2_1}
\end{aligned}
\end{equation}
for suitable positive constants $C$ (independent of the sequence).
Using the continuous Sobolev multiplication $L^{3,\lambda}_{1}\ti L^{4}\ra L^{3,\lambda}$ (which follows from the Sobolev embedding $L^{3,\lambda}_{1}\to L^{12,\lambda}$), together with the Aubin-Sobolev continuous inclusion $L^{2}_{1}\hookrightarrow L^{4}$, it follows that for $i$, $j\geq i_\epsilon$ one has

\begin{equation}\label{ineq1}
\begin{aligned}
\|(\gamma_{j}-\gamma_{i})\sharp(\gamma_{j}+\gamma_{i}-2\gamma_{i_\epsilon})\|_{L^{3,\lambda}}&\leq C\|\gamma_{j}-\gamma_{i}\|_{L^{3,\lambda}_{1}}\|\gamma_{j}+\gamma_{i}-2\gamma_{i_\epsilon}\|_{L^{4}}\\
&\leq C\|\gamma_{j}-\gamma_{i}\|_{L^{3,\lambda}_{1}}\|\gamma_{j}+\gamma_{i}-2\gamma_{i_\epsilon}\|_{L^{2}_{1}}\\
&\leq 2\epsilon C\|\gamma_{j}-\gamma_{i}\|_{L^{3,\lambda}_{1}},
\end{aligned}
\end{equation}
 and 
\begin{equation}\label{ineq2}
\begin{aligned}
\|(\gamma_{j}-\gamma_{i})\sharp 2\gamma_{i_{\epsilon}}\|_{L^{3,\lambda}}&\leq C\|\gamma_{i_\epsilon}\|_{L^{3,\lambda}_{1}}\|\gamma_{i}-\gamma_{j}\|_{L^{4}}\\
&\leq C\|\gamma_{i_\epsilon}\|_{L^{3,\lambda}_{1}}\|\gamma_{i}-\gamma_{j}\|_{L^{2}_{1}}\\
\end{aligned}
\end{equation}
(because  $\gamma_{i_\epsilon}\in L^{2,w}_{k}\hookrightarrow L^{3,\lambda}_{1}$, for $k\geq 2$). Using (\ref{ineqcoerciv2}), we have 

\begin{equation}\label{ineqcoerciv3}
\begin{aligned}
\|\gamma_{i}-\gamma_{j}\|_{L^{3,\lambda}_{1}}&\leq C\|SW(\gamma_{i})-SW(\gamma_{j})\|_{L^{3,\lambda}}+2\epsilon C\|\gamma_{i}-\gamma_{j}\|_{L^{3,\lambda}_{1}}\\
&+ C\|\gamma_{i}-\gamma_{j}\|_{L^{2}_{1}} \|\gamma_{i_\epsilon}\|_{L^{3,\lambda}_{1}}+C\|\gamma_{i}-\gamma_{j}\|_{L^{2}_{1}}.
\end{aligned}
\end{equation}
for all $i$, $j\geq i_\epsilon$, where $C$ is a constant which depends only the Riemannian manifold $(X,g)$, in particular independent of $\epsilon$.  Now choose $\epsilon$ such that  $2\epsilon C=\frac{1}{2}$ and move the second right hand term in (\ref{ineqcoerciv3}) to the left hand side. We get the estimate
\begin{equation}\label{ineqcoerciv4}
\begin{aligned}
\frac{1}{2}\|\gamma_{i}-\gamma_{j}\|_{L^{3,\lambda}_1}&\leq C\|SW(\gamma_{i})-SW(\gamma_{j})\|_{L^{3,\lambda}}+ C\|\gamma_{i}-\gamma_{j}\|_{L^{2}_{1}}.\|\gamma_{i_\epsilon}\|_{L^{3,\lambda}_{1}}\\
&+C\|\gamma_{i}-\gamma_{j}\|_{L^{2}_{1}}.
\end{aligned}
\end{equation}
for any $i,j\geq i_\epsilon$. Now the $L^{2,w}_{k-1}$-bound of $SW(\gamma_{l})$, together with the compact embedding (for $w>\lambda$ and $k\geq 2$),  $L^{2,w}_{k-1}\hookrightarrow L^{3,\lambda}$, imply that $(\gamma_i)_i$ has a subsequence, still denoted in the same way, such that $SW(\gamma_i)$ converges strongly in $L^{3,\lambda}$. Then (\ref{ineqcoerciv4}) shows that $(\gamma_i)_i$ is a Cauchy sequence in $L^{3,\lambda}_1$, hence it is convergent.

{\ }\\
{\it Step 2.}  We prove that the subsequence $(\gamma_i)_i$ obtained in Step 1 is bounded in $L^{2,w}_2$.\\

We follow the same arguments as before. Since $D$ is Fredholm  (for sufficiently small weights), there is a constant $C>0$ (independent of the sequence), such that 
\begin{align*}
\|\gamma_{i}-\gamma_{j}\|_{L^{2,w}_{2}}\leq C(\|D(\gamma_{i}-\gamma_{j})\|_{L^{2,w}_{1}}+\|\gamma_{i}-\gamma_{j}\|_{L^{3,\lambda}_1}).
\end{align*}%
Given $\epsilon$ choose $k_{\epsilon}\in\Nat$ such that for all $i,j\geq k_{\epsilon}$ the inequality $\|\gamma_{i}-\gamma_{j}\|_{L^{3,\lambda}_{1}}\leq \epsilon$ holds.
Then, using  (\ref{Fredholm}), we have for all $i,j\geq k_\epsilon$
\begin{equation}\label{desigualdade2}
\begin{aligned}
\|\gamma_{i}-\gamma_{j}\|_{L^{2,w}_{2}}&\leq C(\|D(\gamma_{i}-\gamma_{j})\|_{L^{2,w}_{1}}+\|\gamma_{i}-\gamma_{j}\|_{L^{3,\lambda}_1})\\
&\leq C\|SW(\gamma_{i})-SW(\gamma_{j})\|_{L^{2,w}_{1}}+C\|(\gamma_{j}-\gamma_{i})\sharp(\gamma_{j}+\gamma_{i}-2\gamma_{k_\epsilon})\|_{L^{2,w}_{1}}+\\
&  +C\|(\gamma_{j}-\gamma_{i})\sharp 2\gamma_{k_\epsilon}\|_{L^{2,w}_{1}}+C\|\gamma_{i}-\gamma_{j}\|_{L^{3,\lambda}_1}.
\end{aligned}
\end{equation}
The continuous Sobolev multiplication $L^{2,w}_{2}\ti L^{3,\lambda}_{1}\ra L^{2,w}_{1}$ (which follows from both embeddings $L^{2,w}_{2}\ra L^{4,w}_{1}$ and $L^{3,\lambda}_{1}\ra L^{12,\lambda}$) implies that, for all $i,j\geq k_\epsilon$
\begin{align*}
\|(\gamma_{j}-\gamma_{i})\sharp(\gamma_{j}+\gamma_{i}-2\gamma_{k_\epsilon})\|_{L^{2,w}_{1}}&\leq C\|\gamma_{j}-\gamma_{i}\|_{L^{2,w}_{2}}\|\gamma_{j}+\gamma_{i}-2\gamma_{k_\epsilon}\|_{L^{3,\lambda}_{1}}\\
&\leq 2\epsilon C\|\gamma_{j}-\gamma_{i}\|_{L^{2,w}_{2}},
\end{align*}
and
\begin{align*}
\|(\gamma_{j}-\gamma_{i})\sharp 2\gamma_{k_\epsilon}\|_{L^{2,w}_{1}}&\leq C\|\gamma_{k_\epsilon}\|_{L^{2,w}_{2}}\|\gamma_{i}-\gamma_{j}\|_{L^{3,\lambda}_{1}}
\end{align*}
(because $\gamma_{k_\epsilon}\in L^{2,w}_{k}\hookrightarrow L^{2,\lambda}_{2}$, for $k\geq 2$). Using (\ref{desigualdade2}), we have
\begin{equation}\label{desigualdade4}
\begin{aligned}
\|\gamma_{i}-\gamma_{j}\|_{L^{2,w}_{2}}&\leq C\|SW(\gamma_{i})-SW(\gamma_{j})\|_{L^{2,w}_{1}}+2\epsilon C\|\gamma_{i}-\gamma_{j}\|_{L^{2,w}_{2}}\\
&+C\|\gamma_{i}-\gamma_{j}\|_{L^{3,\lambda}_{1}}\|\gamma_{k_\epsilon}\|_{L^{2,w}_{1}}+C\|\gamma_{i}-\gamma_{j}\|_{L^{3,\lambda}_{1}}
\end{aligned}
\end{equation}
where $C$ is a constant which depends only the Riemannian manifold $(X,g)$, in particular independent of $\epsilon$.  Now, as before, choose $\epsilon$ such that  $2\epsilon C=\frac{1}{2}$ and move the second right hand term in (\ref{desigualdade4}) to the left hand side. We get the estimate
\begin{equation}\label{desigualdade5}
\begin{aligned}
\frac{1}{2}\|\gamma_{i}-\gamma_{j}\|_{L^{2,w}_{2}}&\leq C\|SW(\gamma_{i})-SW(\gamma_{j})\|_{L^{2,w}_{1}}\\
&+ C\|\gamma_{k_\epsilon}\|_{L^{2,w}_{1}}\|\gamma_{i}-\gamma_{j}\|_{L^{3,\lambda}_{1}}+C\|\gamma_{i}-\gamma_{j}\|_{L^{3,\lambda}_{1}}
\end{aligned}
\end{equation}
for all $i,j\geq k_\epsilon$. Since $\|SW(\gamma_{j})\|_{L^{2,w}_{1}}$ is bounded and $(\gamma_i)_i$ is Cauchy in $L^{3,\lambda}_{1}$, it follows that $(\gamma_i)_i$ is bounded in $L^{2,w}_{2}$.

{\ }\\
{\it Step 3.}  Supposing $k\geq 3$,  we prove that the subsequence obtained in Step 1, $(\gamma_i)_i$,  has itself a subsequence converging strongly in $L^{2,\lambda}_2$.  \\

We follow the same arguments. Since $D$ is Fredholm, there is a constant $C>0$ (independent of the sequence), such that 
\begin{align*}
\|\gamma_{i}-\gamma_{j}\|_{L^{2,\lambda}_{2}}\leq C(\|D(\gamma_{i}-\gamma_{j})\|_{L^{2,\lambda}_{1}}+\|\gamma_{i}-\gamma_{j}\|_{L^{3,\lambda}_1}).
\end{align*}%
Given $\epsilon$ choose $l_{\epsilon}\in\Nat$ such that for all $i,j\geq l_{\epsilon}$ the inequality $\|\gamma_{i}-\gamma_{j}\|_{L^{3,\lambda}_{1}}\leq \epsilon$ holds.
Then, using  (\ref{Fredholm}), we have for all $i,j\geq l_\epsilon$
\begin{equation}\label{desigualdade3}
\begin{aligned}
\|\gamma_{i}-\gamma_{j}\|_{L^{2,\lambda}_{2}}&\leq C(\|D(\gamma_{i}-\gamma_{j})\|_{L^{2,\lambda}_{1}}+\|\gamma_{i}-\gamma_{j}\|_{L^2_1})\\
&\leq C\|SW(\gamma_{i})-SW(\gamma_{j})\|_{L^{2,\lambda}_{1}}+C\|(\gamma_{j}-\gamma_{i})\sharp(\gamma_{j}+\gamma_{i}-2\gamma_{l_\epsilon})\|_{L^{2,\lambda}_{1}}+\\
&  +C\|(\gamma_{j}-\gamma_{i})\sharp 2\gamma_{l_\epsilon}\|_{L^{2,\lambda}_{1}}+C\|\gamma_{i}-\gamma_{j}\|_{L^{3,\lambda}_1}.
\end{aligned}
\end{equation}
The continuous Sobolev multiplication $L^{2,\lambda}_{2}\ti L^{3,\lambda}_{1}\ra L^{2,\lambda}_{1}$ (which follows from both embeddings $L^{2,\lambda}_{2}\ra L^{4,\lambda}_{1}$ and $L^{3,\lambda}_{1}\ra L^{12,\lambda}$) implies that 
\begin{align*}
\|(\gamma_{j}-\gamma_{i})\sharp(\gamma_{j}+\gamma_{i}-2\gamma_{l_\epsilon})\|_{L^{2,\lambda}_{1}}&\leq C\|\gamma_{j}-\gamma_{i}\|_{L^{2,\lambda}_{2}}\|\gamma_{j}+\gamma_{i}-2\gamma_{l_\epsilon}\|_{L^{3,\lambda}_{1}}\\
&\leq 2\epsilon C\|\gamma_{j}-\gamma_{i}\|_{L^{2,\lambda}_{2}},
\end{align*}
and also
\begin{align*}
\|(\gamma_{j}-\gamma_{i})\sharp 2\gamma_{l_\epsilon}\|_{L^{2,\lambda}_{1}}&\leq C\|\gamma_{l_\epsilon}\|_{L^{2,\lambda}_{2}}\|\gamma_{i}-\gamma_{j}\|_{L^{3,\lambda}_{1}}\\
\end{align*}
(because $\gamma_{l_\epsilon}\in L^{2,w}_{k}\hookrightarrow L^{2,\lambda}_{2}$, for $k\geq 2$. Using (\ref{desigualdade3}), we have
\begin{equation}\label{ineqcoerciv5}
\begin{aligned}
\|\gamma_{i}-\gamma_{j}\|_{L^{2,\lambda}_{2}}&\leq C\|SW(\gamma_{i})-SW(\gamma_{j})\|_{L^{2,\lambda}_{1}}+2\epsilon C\|\gamma_{i}-\gamma_{j}\|_{L^{2,\lambda}_{2}}\\
&+C\|\gamma_{i}-\gamma_{j}\|_{L^{3,\lambda}_{1}}\|\gamma_{l_\epsilon}\|_{L^{2,\lambda}_{1}}+C\|\gamma_{i}-\gamma_{j}\|_{L^{3,\lambda}_{1}}
\end{aligned}
\end{equation}
where $C$ is a constant which depends only the Riemannian manifold $(X,g)$, in particular is independent of $\epsilon$.  Now, as before, choose $\epsilon$ such that  $2\epsilon C=\frac{1}{2}$ and move the second right hand term in (\ref{ineqcoerciv5}) to the left hand side. We get the estimate 
\begin{equation}\label{ineqcoerciv6}
\begin{aligned}
\frac{1}{2}\|\gamma_{i}-\gamma_{j}\|_{L^{2,\lambda}_{2}}&\leq C\|SW(\gamma_{i})-SW(\gamma_{j})\|_{L^{2,\lambda}_{1}}\\
&+ \epsilon C\|\gamma_{l_\epsilon}\|_{L^{2,\lambda}_{2}}+C\|\gamma_{i}-\gamma_{j}\|_{L^{3,\lambda}_{1}} ,
\end{aligned}
\end{equation}
Now the $L^{2,w}_{k-1}$-bound of $SW(\gamma_{l})$, together with the compact inclusion (for $k\geq 3$) $L^{2,w}_{k-1}\hookrightarrow L^{2,\lambda}_{1}$, imply that $(\gamma_i)_i$ has a subsequence, still denoted the same way, such that $SW(\gamma_i)$ converges strongly in $L^{2,\lambda}_{1}$. Then from (\ref{ineqcoerciv6}), it follows that $(\gamma_i)_i$ is a Cauchy sequence in $L^{2,\lambda}_{2}$, hence it is convergent in $L^{2,\lambda}_2$.\\

{\it Step 4.} 
Let  $(\gamma_{i})_i$ be a sequence converging strongly in $L^{2,\lambda}_{k-1}$. 
Suppose that $\|SW(\gamma_i)\|_{L^{2,w}_{k}}\leq C$ for all $i\in\Nat$. Then $(\gamma_i)_i$  is bounded in  $L^{2,w}_{k}$.\\

The proof of this statement follows the same lines as Step 2, i.e., the decomposition (\ref{Fredholm}),  Fredholmness of $D$, the Sobolev multiplication $L^{2,w}_{k}\ti L^{2,\lambda}_{k-1}\ra L^{2,w}_{k-1}$ but in the final step, it uses only the  $L^{2,w}_{k}$-bound on $SW(\gamma_i)$, instead of a compact inclusion argument. \\

{\ }\\
{\it Step 5.}
Let  $(\gamma_{i})_i$ be a sequence converging strongly in $L^{2,\lambda}_{s-1}$. Suppose that \linebreak $\|SW(\gamma_i)\|_{L^{2,w}_{s}}\leq C$ for all $i\in\Nat$. Then $(\gamma_i)_i$ has a subsequence converging strongly in $L^{2,\lambda}_{s}$.\\

This statement is proven using the same arguments as Step 3, namely, the decomposition (\ref{Fredholm}), Fredholmness of $D$, the continuous Sobolev multiplication $L^{2,\lambda}_{s}\ti L^{2,\lambda}_{s-1}\ra L^{2,\lambda}_{s-1}$  and, finally, the  $L^{2,w}_{s}$-bound on $SW(\gamma_i)$ together with the compact inclusion $L^{2,w}_{s}\ra L^{2,\lambda}_{s-1}$. \\ 

Now we come back to the proof of Theorem \ref{coerc}: 

1. For $k=2$ the claim is proved in Step 2.  $k\geq 3$ we use recursively Step 5, and then Step 4.

2. For $k=2$ the claim follows from Step 1, taking into account that $L^{3,\lambda}_{1}$ is continuously embedded in $L^{2,\lambda}_{1}$. 
For $k=3$ this is Step 3, and for $k\geq 3$ we use recursively Step 5. 
\end{proof}

\begin{re}
A similar proof to the one before holds for the perturbed Seiberg-Witten equation.
\end{re}
\vspace{ 1 pc}

\begin{coro}
Suppose we are in the conditions of Theorem \ref{coerc}. 

If $(\psi,\chi)\in L^{2,w+\varepsilon}_{k+1}(\Sigma^-)\times L^{2,w+\varepsilon}_{k+1}(\Herm_0(\Sigma^+))$, then  the fiber
$$\big\{(v,\phi)\in d^*(L^{2,w}_{k+1}(i\Lambda^2_+))\times  L^{2,w}_k(\Sigma^+)|\ SW(A_0+v, \phi)=(\psi,\chi)\big\}
$$
is compact in $L^{2,w}_k$.
\end{coro}

In the next section we will see that (adding to admissibility the technical condition $H_1(W_e,\Z)$ is torsion for any end $e$) the Seiberg-Witten map induces a map between Hilbert bundles over the torus $B:=iH^1(X,\R)/H^1(X,2\pi i \Z)$. Theorem \ref{coerc} gives the fiberwise coercivity of this map. The  global coercivity is given by the following simple generalization of Theorem \ref{coerc}, which can be proved following step by step the same  arguments:

\begin{thh}\label{coerc-new}
Let $(X,g,\tau,A_0)$ be an admissible 4-tuple.
Suppose that $w>0$ and $k\geq 2$.
Fix a compact set $\Pi\subset\mathbb{H}^1_w$

Let $(v_{n},h_n,\varphi_{n})_{n\in\mathbb{N}}$ be a sequence in  
$$\textrm{Im }[d^{*}:L^{2,w}_{k+1}(\Lambda^{2}_{+})\ra L^{2,w}_{k}(\Lambda^{1})]\ti\Pi\ti L^{2,w}_{k}(\Sigma^{+})$$
such that there exists $C>0$ with
\begin{equation}\label{bound}
\begin{aligned}
\|SW(v_{n}+h_n,\varphi_{n})\|_{L^{2,w}_{k-1}}\leq C \ \forall n\in \mathbb{N}\ .
\end{aligned}
\end{equation}

\begin{enumerate}
\item   There exists a subsequence of $ (v_{n},\varphi_{n})_{n\in\mathbb{N}}$  which is bounded in $L^{2,w}_{k}$,
\item    For any $\lambda\in (0, w)$ there exists a subsequence of $ (v_{n},\varphi_{n})_{n\in\mathbb{N}}$  which is strongly convergent  in $L^{2,\lambda}_{k-1}$. \end{enumerate}
\end{thh}

The compactness of the Seiberg-Witten moduli space follows from

\begin{coro}
Suppose we are in the conditions of Theorem \ref{coerc-new}. 

If $(\psi,\chi)\in L^{2,w+\varepsilon}_{k+1}(\Sigma^-)\times L^{2,w+\varepsilon}_{k+1}(\Herm_0(\Sigma^+))$, then  the fiber
$$\big\{(v,h,\phi)\in d^*(L^{2,w}_{k+1}(i\Lambda^2_+))\times \Pi \times  L^{2,w}_k(\Sigma^+)|\ SW(A_0+v, \phi)=(\psi,\chi)\big\}
$$
is compact in $L^{2,w}_k$.
\end{coro}

\chapter{Cohomotopy invariants}

\section{Cohomotopy invariants}

 We will follow the formalism introduced in \cite{OT}  (sections 3.1 and 3.3).
 Let $\E$ and $\F$ be a pair of Hermitian Hilbert vector bundles over B a compact Hausdorff space and let $\V$ and $\W$ be a pair of real Hilbert vector spaces. Note that there is a natural $S^1$-action on $\E\ti\V$ and $\F\ti\W$.

Suppose $\mu:\E\ti\V\ra \F\ti\W$ is an $S^{1}$-equivariant map over B such that
\begin{enumerate}
\item
$\mu$ is fiberwise differentiable at the origin of each fiber with its fiberwise differential continuous on $\E\ti\V$,
\item
the fiberwise differentials at the origin
\begin{align*}
D_{b}=d_{0_{b}}\mu_{b}:\E_{b}\ti\V\ra\F_{y}\ti\W
\end{align*}
are Fredholm, and that the linear operator $D_{b}$ has the form $D_{b}=(\Xi_{b},L_{b})$, where $\Xi_{b}:\E_{b}\ra\F_{b}$ 
and $L_{b}:\V\ra\W$, are the Fredholm maps defined by the derivatives of the restrictions $\mu_{|_{\E_{b}\ti {0^{\V}}}}$ and $\mu_{{|}_{0^{\E}_{b}\ti\V}}$. 
\end{enumerate}
Note that $\Xi$ defines a family of complex operators over B, and, therefore, determines a class $\mathrm{ ind }(\Xi)\in K(B)$.

\begin{de} 
An $S^1$-equivariant map $\mu:\E\ti \V\ra \F\ti \W$ over $B$, verifiying conditions 1) and 2) above, is said to be 
\begin{enumerate}
\item coercive, if for every $c>0$, there exists $C_c>0$, such that any $(e,v)\in\E\ti\V$ with $\|\mu(e,v)\|\leq c $ has to verifiy $\|(e,v)\|\leq C_c$,
\item admissible if, moreover, the following holds:
\begin{enumerate}
\item
There is an orthogonal decomposition $\W=H\oplus \W_{0}$, where H is finite dimensional subspace and   $\mu(0_{b}^{\E},v)=\kappa(b)+L(v)$. Here $L:\V\ra \W_{0}$ induces a linear isometry and $\kappa:B\ra H$ is nowhere vanishing map
\item
The above map $\kappa:B\ra H\subseteq \W$ is nowhere vanishing on B. (Existing then $\epsilon$ such that $ \|\kappa(b)\|=\|\pi_{H}\mu(0^{E}_{b},v)\|\geq 0. $)
\item
$\mu-D$ is compact i.e. for every $R>0$ the image 
$$ \mathrm{Im}[\mu-D: \mathrm{Disk}_{R}(\E\ti\V)\ra \F\ti\W] $$
 is relatively compact in the total space $\F\ti\W$. 
\end{enumerate}
\end{enumerate}
\end{de}

In   \cite{OT} the authors construct  a graded cohomotopy group $\alpha^*(\mathrm{ ind }(\Xi))$ and, in the presence of a admissible map $\mu$, and invariant
$$\{\mu\}\in \alpha^{\dim(H)-1}(\mathrm{ ind }(\Xi)) .
$$

\section{The Seiberg-Witten cohomotopy map}

Let $(X,g,\tau,A_0)$ be an admissible 4-tuple such that $H_{1}(W_e,\Z)$ is torsion for every end $e\in\Eg$. 

Following the ideas explained in section \ref{intro2} put
$$\A_0:=A_0+\mathbb{H}^1_w\ .
$$

In section \ref{gaugegroup}, we introduced the group

$$G:=\{\theta\in\mathscr{G}_{w,k+1}|\ \theta^{-1} d\theta\in \mathbb{H}^1_w\},$$
which fits the short exact sequence
$$
0\ra S^1\ra G\ra H^{1}(X,2\pi i\Z)\ra 0.
$$

Fix a point $x_0\in X$ and let $G_{x_0}\simeq H^{1}(X,2\pi i\Z) $ be the kernel of the evaluation map 
$$\mathrm{ev}_{x_0}:G\ra S^1.$$ 

The group $G_{x_0}$ acts freely on $\A_0$ and the quotient $\A_{0}/G_{x_0}$ can be identified with the quotient $\mathbb{H}^1_w/G_{x_0}$.
Using Proposition \ref{kerp} (4) we see that this quotient can be identified with $H^{1}(X,i\R)/H^{1}(X,2\pi i\Z)$. By Corollary \ref{tensor}, the group $H^{1}(X,2\pi i\Z)$ is a lattice in  the finitely dimensional vector space $ H^{1}(X,i\R)$. Note that, although very natural, this statement is not at all obvious. Therefore the free quotient
$$B:=\A_0/G_0
$$
is a torus of dimension $\dim(H^{1}(X,i\R))$.  We put now as in section \ref{intro2} 
\begin{align*}
\mathscr{E}&=\A_0 \ti_{G_{x_0}}L^{2,w}_{k}(\Sigma^{+})&\V&=\textrm{Im }[d^{*}:L^{2,w}_{k+1}(\Lambda^{2}_{+})\ra L^{2,w}_{k}(\Lambda^1)]\\
\mathscr{F}&=\A_0 \ti_{G_{x_0}}L^{2,w}_{k-1}(\Sigma^{-})&\W&=L^{2,w}_{k}(i\Lambda^{2}_{+})\ ,
\end{align*}
and we see that $\E$, $\F$ are Hilbert bundles over the torus $B$ and that the    Seiberg-Witten map $SW$ induces an $S^1$-equivariant bundle map

\begin{equation}\label{SWoverT}
\begin{array}{c}
\unitlength=1mm
\begin{picture}(20,12)(-5,-4)
\put(-6,4){$\V\times \E$}
\put(5,5){\vector(2,0){10}}
\put(16,4){$ \W\times \F$}
\put(7.5,6.5){$\mu_{SW}$}
%\put(2,2){$\simeq$}
\put(2,2){\vector(2, -3){5}}
\put(9,-8){$B$}
%\put(-4,-3){$\lambda$}
\put(18,2){\vector(-2, -3){5}}
%\put(16,-8){.}
\end{picture} 
\end{array} 
\end{equation}
\vspace{2mm}
over the torus $B$. 

Our main result Theorem \ref{coerc-new} shows that 
\begin{thh}
Let $(X,g,\tau,A_0)$ be an admissible 4-tuple, such that $H_{1}(W_e,\Z)$ is torsion for every end $e\in\Eg$. Then the Seiberg-Witten $\mu_{SW}$ over the torus $B$ is coercive.
\end{thh}

\begin{re}
The same property holds for the perturbed Seiberg-Witten map $\mu_{SW_\eta}$. Moreover, assuming $\dim (\mathbb{H}^{+}_{2})>0$, then $\mu_{SW_\eta}$ is admissible for a generic perturbation $\eta$. Therefore a cohomotopy invariant $\{\mu_{SW_\eta}\}$ can be defined.  
\end{re} 

Note however that a priori this invariant might depend on $\eta$, $A_0$, $g$ and even on the fixed periodic end structure on $X$. The dependence of $\{\mu_{SW_\eta}\}$ in these parameters will be studied in a future work.

\appendix

\chapter{}

\begin{lemma}
Suppose $(X,g)$ is an oriented 4 manifold with bounded geometry. Then the composition $I:L^{2}_{1}(\Lambda^{1})\ti L^{2}_{1}(\Lambda^{2})\ra L^{1}(\Lambda^{4})\ra\R$ given by
\begin{align*}
I(u,v)=\int_{X} du\wedge v-u\wedge dv
\end{align*}
is identically zero.
\end{lemma}

\begin{proof}
The result is trivial for smooth forms with compact support, hence it follows for general forms by continuous bilinearity of $I$ together with the density of the space of smooth forms with compact support.
\end{proof}

\begin{lemma}\label{kerd}
Suppose $(X,g)$ is an oriented 4-manifold with bounded geometry. Then $d:L^{2}_2(\Lambda^{1})\ra L^{2}_1(i\Lambda^{2})$ and $d^{+}=\pi_{\Lambda^{2}_{+}}\circ d$.  
Then $\ker d^{+}=\ker d$.
\end{lemma}

\begin{proof}
Let $\alpha\in L^{2}_2(i\Lambda^{1})$ be such that $d^{+}\alpha=0$. Then $\beta=d\alpha\in L^{2}_1(i\Lambda^{1})$ is a self dual $2$-form.
By the previous lemma we have
\begin{align*}
0&=\int_{X} d\alpha\wedge \beta-\alpha\wedge d\beta=\int_{X} \beta\wedge\beta \\
&=\int_{X} \beta\wedge \ast \beta=\|\beta\|^{2}_{L^{2}}.
\end{align*}
\end{proof}

\begin{thh}\label{inclusion}
Let $(X,g)$ be a Riemannian manifold with periodic ends of dimension $n$ . Suppose that
 \begin{enumerate}
\item $k - \bar{k}\geq  n/p -n/\bar{p}$ , 
\item$k\geq\bar{k}\geq 0$ and either
\item  $1<p\leq \bar{p}<\infty$  with $\bar{w}\leq w$ or
\item $ 1<\bar{p}<p<\infty$  with $\bar{w}<w$. 
\end{enumerate}
Then there is a continuous embedding $L^{p,w}_{k}(E)\ra L^{\bar{p},\bar{w}}_{\bar{k}}(E)$.
\end{thh}

\begin{proof}
This follows along the lines of the proof of lemma 7.2 in \cite{LM}, page 435 (see also  3.10 in \cite{L}, page 14).
\end{proof}

\begin{thh}\label{rellichk}
Let $(X,g)$ be a Riemannian manifold with periodic ends of dimension $n$ . Suppose that
 \begin{enumerate}
\item $k - \bar{k}> n/p -n/\bar{p}$ , 
\item$k>\bar{k}\geq 0$ and
\item $\bar{w} < w$.
\end{enumerate}
Then the embedding $L^{p,w}_{k}(E)\ra L^{\bar{p},\bar{w}}_{\bar{k}}(E)$  is compact.
\end{thh}

\begin{proof}
This follows along the lines of the proof of theorem 3.12 in \cite{L}, page 15.
\end{proof}

\begin{coro}\label{compa}
Let $(X,g)$ be a Riemannian 4-manifold with periodic ends.
Then for $\bar{w}<w$  and $k\geq 1$ we have the following compact injections
\begin{align*}
L^{2,w}_{k}&\hookrightarrow L^{3,\bar{w}}_{k-1}    \\
L^{2,w}_{k}&\hookrightarrow L^{2,\bar{w}}_{k-1}     \\
L^{2,w}_{k}&\hookrightarrow L^{2,\bar{w}}_{k-1}      \\
\end{align*}
 \end{coro}

\begin{prop}\label{multi}
$(X,g)$ be a Riemannian manifold with periodic ends. The following continuous Sobolev multiplications hold:
\begin{align*}
L^{4,a}\ti L^{3,b}_{1}\ra L^{3,a+b}\\
L^{2,a}_{2}\ti L^{3,b}_{1}\ra L^{2,a+b}_{1} \\
L^{2,a}_{3}\ti L^{2,b}_{2}\ra L^{2,a+b}_{2}
\end{align*}
\end{prop}

\begin{proof}
It is sufficient to prove the above inclusions without the weights.

\item 1) \underline{$L^{4}\ti L^{3}_{1}\ra L^{3}$}

From  (\ref{inclusion}), we have $L^{2}_{1}\hookrightarrow L^{4}$ and $L^{3}_{1}\hookrightarrow L^{12}$.
From H\"older and the equality $1/4+1/12 =1/3$ it follows
\begin{align*}
\|uv\|_{L^3}\leq\|u\|_{L^4}\|v\|_{L^{12}} \leq \|u\|_{L^{2}_{1}}\|v\|_{L^{3}_{1}}
\end{align*}

\item 2) \underline{$L^{2}_{2}\ti L^{3}_{1}\ra L^{2}_{1}$}

From  (\ref{inclusion}), we have $L^{2}_{2}\hookrightarrow L^{4}_{1}\cap L^6$ and $L^{3}_{1}\hookrightarrow{L^4}$. It follows
\begin{align*}
\nabla (uv) =\nabla u\cdot v+u\cdot\nabla v \in L^2
\end{align*}
using $\nabla u \in L^{4}$ and  $v\in L^4$ , $u\in L^6$ and $\nabla v\in L^{3}$.  From $u\in L^6$ and $v\in L^{3}$ it follows $uv\in L^2$

\item 3) \underline{$L^{2}_{3}\ti L^{2}_{2}\ra L^{2}_{2}$}

Follows similarly.
\end{proof}

\end{document}